\documentclass[12pt,leqno,twoside,sort]{amsart}

\usepackage{amssymb,amsmath,amsthm,soul,color}
\usepackage[normalem]{ulem}
\usepackage{amssymb,amsmath,amsthm}
\usepackage{mathrsfs}
\usepackage[utf8]{inputenc}
\usepackage{todonotes}
\usepackage[left=2cm, right=2cm, top=2cm, bottom=2cm]{geometry}
\usepackage{url}
\usepackage[numbers,sort&compress]{natbib}
\usepackage{fixmath}
\usepackage{bm}

\usepackage[T1]{fontenc}
\usepackage{lmodern}
\usepackage{microtype}

\usepackage{mathtools}

\usepackage[colorlinks=true,urlcolor=blue,
citecolor=red,linkcolor=blue,linktocpage,pdfpagelabels,pagebackref=true,bookmarksnumbered,bookmarksopen]{hyperref}

\usepackage{xcolor,cancel}

\usepackage{enumitem,graphicx}

\usepackage[mathlines]{lineno}

\usepackage{aliascnt}

\usepackage[capitalise,noabbrev,nameinlink]{cleveref}

\theoremstyle{theorem}
\newtheorem{Th}{Theorem}[section]
\crefname{Th}{Theorem}{Theorems}

\newaliascnt{Prop}{Th}
\newtheorem{Prop}[Prop]{Proposition}
\aliascntresetthe{Prop}
\crefname{Prop}{Proposition}{Propositions}

\newaliascnt{Lem}{Th}
\newtheorem{Lem}[Lem]{Lemma}
\aliascntresetthe{Lem}
\crefname{Lem}{Lemma}{Lemmas}

\newaliascnt{Cor}{Th}

\aliascntresetthe{Cor}
\crefname{Cor}{Corollary}{Corollaries}

\theoremstyle{definition}
\newaliascnt{Def}{Th}

\aliascntresetthe{Def}
\crefname{Def}{Definition}{Definitions}

\theoremstyle{remark}
\newaliascnt{Rem}{Th}
\newtheorem{Rem}[Rem]{Remark}
\aliascntresetthe{Rem}
\crefname{Rem}{Remark}{Remarks}

\newaliascnt{Ex}{Th}

\aliascntresetthe{Ex}
\crefname{Ex}{Example}{Examples}

\newaliascnt{Ass}{Th}

\aliascntresetthe{Ex}
\crefname{Ex}{Assumption}{Assumptions}

\newaliascnt{Que}{Th}

\aliascntresetthe{Que}
\crefname{Que}{Question}{Questions}

\crefname{subsection}{Subsection}{Subsections}

\newcommand{\wt}{\widetilde}

\newcommand{\R}{\mathbb{R}}

\newcommand{\Z}{\mathbb{Z}}
\newcommand{\N}{\mathbb{N}}

\newcommand{\cC}{{\mathcal C}}

\newcommand{\cH}{{\mathcal H}}

\newcommand{\cM}{{\mathcal M}}

\newcommand{\cO}{{\mathcal O}}

\newcommand{\Ga}{\Gamma}

\newcommand{\weakto}{\rightharpoonup}

\renewcommand{\div}{\mathrm{div}\,}

\newcommand{\tu}{\widetilde{u}}

\numberwithin{equation}{section}

\newcommand{\supp}{\mathrm{supp}\,}

\begin{document}

\title{Ground state solutions to the nonlinear Born-Infeld problem}

\author[B. Bieganowski]{Bartosz Bieganowski}
\address[B. Bieganowski]{\newline\indent
			Faculty of Mathematics, Informatics and Mechanics, \newline\indent
			University of Warsaw, \newline\indent
			ul. Banacha 2, 02-097 Warsaw, Poland}	
			\email{\href{mailto:bartoszb@mimuw.edu.pl}{bartoszb@mimuw.edu.pl}}	

\author[N. Ikoma]{Norihisa Ikoma}
\address[N. Ikoma]{\newline\indent  	
			Department of Mathematics,		\newline\indent 
			Faculty of Science and Technology, \newline\indent 
			Keio University, \newline\indent
			Yagami Campus: 3-14-1 Hiyoshi, Kohoku-ku, Yokohama, Kanagawa 2238522, Japan}
			\email{\href{mailto:ikoma@math.keio.ac.jp}{ikoma@math.keio.ac.jp}}
			
\author[J. Mederski]{Jarosław Mederski}
\address[J. Mederski]{\newline\indent  	
			Institute of Mathematics,		\newline\indent 
			Polish Academy of Sciences, \newline\indent 
			ul. \'Sniadeckich 8, 00-656 Warsaw, Poland
		\newline\indent
		and
		\newline\indent 
		Faculty of Mathematics and Computer Science,		\newline\indent 
		Nicolas Copernicus University, \newline\indent ul. Chopina 12/18,		 87-100 Toruń, Poland}
			\email{\href{mailto:jmederski@impan.pl}{jmederski@impan.pl}}

\date{}	\date{\today}

\begin{abstract} 
	In the paper we show the existence of ground state solutions to the nonlinear Born-Infeld problem
\[
\div \left( \frac{\nabla u}{\sqrt{1-|\nabla u|^2}} \right) + f(u) = 0, \quad x \in \R^N
\]
in the zero and positive mass cases. Moreover, we find a new proof of the Sobolev-type inequality
\[
\int_{\R^N} \left(1 - \sqrt{1-|\nabla u|^2}\right) \, dx \geq  C_{N,p} \left( \int_{\R^N} |u|^p \, dx \right)^{\frac{N}{N+p}},
\]
for $p > 2^*$ as well as the characterization of the optimal constant $C_{N,p}$ in terms of the ground state energy level. 
Previous approaches relied on approximation schemes and/or symmetry assumptions, 
which typically yield to compact embeddings and may lead to solutions that are not at the ground state energy level. 
In contrast, neither approximation arguments nor symmetry assumptions are employed in the paper to obtain a ground state solution. 
Instead, we develop a new direct variational approach based on minimization over a Pohožaev manifold combined with profile decomposition techniques. 
Finally, we show that nonradial solutions exist whenever 
$N \geq 4$; in particular, this settles a previously open problem in the case $N=5$.

\medskip

   
\end{abstract} 

\maketitle

\pagestyle{myheadings} \markboth{\underline{B. Bieganowski, N. Ikoma, J. Mederski}}{
		\underline{Ground state solutions to the nonlinear Born-Infeld problem}}

\section{Introduction}

%


%

In this paper, we consider the nonlinear problem
\begin{equation}\label{eq:BI}
\div \left( \frac{\nabla u}{\sqrt{1-|\nabla u|^2}} \right) + f(u) = 0, \quad x \in \R^N.
\end{equation}
The problem \eqref{eq:BI} appears when one considers the electromagnetic theory introduced by Born and Infeld \cite{Born1, Born2, Born3, Born4} as a nonlinear alternative to the Maxwell theory. The new electromagnetic theory solves the well-known \textit{infinite energy problem}. Namely, if we consider the electromagnetic field
generated by a point charge in $\R^3$, and the corresponding scalar field equation in the Maxwell theory
\[
\Delta \phi + \delta_0 = 0, \quad x \in \R^3,
\]
then the \textit{fundamental solution}
\[
\phi(x) = \frac{C}{|x|}
\]
where $C > 0$, has an infinite energy 
\[
E_M(\phi) = \int_{\R^3} |\nabla \phi(x)|^2 \, dx = +\infty.
\]
In the Born-Infeld theory the scalar field equation is governed by the operator $\div \left( \frac{\nabla \phi}{\sqrt{1-|\nabla \phi|^2}} \right)$. Then the fundamental solution $\phi(x) = \phi(|x|)$ to
\[
\div \left( \frac{\nabla \phi}{\sqrt{1-|\nabla \phi|^2}} \right) + \delta_0 = 0, \quad x \in \R^3,
\]
is given by the derivative of its radial profile with some $b>0$
\[
\phi'(r) = \frac{b^2}{\sqrt{r^4 + b^2}}
\]
and its energy
\[
E_{BI} (\phi)=  \int_{\R^3} \left(1 - \sqrt{1-|\nabla \phi(x)|^2} \right) \, dx = 4 \pi \int_0^\infty \left(1 - \frac{r^2}{\sqrt{r^4+b^2}} \right) r^2 \, dr < \infty.
\]

We remark that \eqref{eq:BI} has variational structure. In fact, the energy functional associated to \eqref{eq:BI} is given by
\begin{equation}\label{eq:I}
	I(u) := \int_{\R^N} \left(1 - \sqrt{1-|\nabla u|^2}\right) \, dx - \int_{\R^N} F(u) \, dx = \Psi(u) - \Phi(u),
\end{equation}
where we set
\[
F(u) := \int_0^u f(s) \, ds, \quad 
\Psi(u) := \int_{\R^N} \left(1 - \sqrt{1-|\nabla u|^2}\right) \, dx, \quad \Phi(u) := \int_{\R^N} F(u) \, dx.
\]
Formally, if $\| \nabla u \|_\infty = \| \nabla u \|_{L^\infty(\R^N)} < 1$, then 
\[
\lim_{h \to 0} \frac{I(u+t \varphi) - I(u)}{h} = \int_{\R^N} \frac{\nabla u \cdot \nabla \varphi}{\sqrt{ 1 - | \nabla u|^2 }} - f(u) \varphi \, dx
\]
and $u$ is a weak solution of \eqref{eq:BI} provided $u$ is a critical point of $I$ with $\|\nabla u \|_{\infty} < 1$. 
However, it should be pointed out that for the well-definedness of $\Psi$, the condition $|\nabla u| \leq 1$ for a.e. $x \in \R^N$ is required and 
it is not possible to find a function space on which $\Psi$ becomes of class $\cC^1$. 
Therefore, we need to go beyond the smooth critical point theory to treat $I$ directly on some function space, 
and the nonsmoothness of $\Psi$ makes the problem complicated.

The problem
\[
\div \left( \frac{\nabla u}{\sqrt{1-|\nabla \phi|^2}} \right) + \rho = 0
\]
with fixed $\rho$ has been studied by many authors; see, for example,
\cite{BCF, BdAP, BdAPR, BI1, BI2, BIMM24, H, K1, K2}.
On the other hand, the literature for \eqref{eq:BI} is fragmentary. 
In \cite{BoDeCoDe12}, the power nonlinearities $f(u) = |u|^{p-2}u$ with $p>2^*$ 
are studied and the existence of radial positive classical solution and infinitely many radial solutions is proved 
through approximations of the differential operator. Furthermore, the following Sobolev type inequality is 
established and the existence of optimizers is studied: 
\begin{equation}\label{Sob-type-ineq}
	\int_{\R^N} \left( 1 - \sqrt{1 - |\nabla u|^2} \right) dx \geq C_{N,p} \left( \int_{\R^N} |u|^p \, dx \right)^{ \frac{N}{N+p} }.
\end{equation}
In \cite{Azz14}, \eqref{eq:BI} is studied for $N \geq 2$ and $f(u) = - u + g(u)$ by the shooting method, 
and the existence of radial positive decreasing solution is found under some assumptions on $g$. 
On the other hand, \cite{Azz16} treats the case $N \geq 3$ and $f(u) = |u|^{p-2} u$ with $ p > 2$ and $p \neq 2^*$ where $2^* := 2N/(N-2)$. 
It is proved that for $p \in (2,2^*)$, \eqref{eq:BI} has no radial solution decaying at infinity, and 
for $p> 2^*$, \eqref{eq:BI} has infinitely many radial solutions with infinite energy, that is $\| \nabla u \|_2 = \infty$. 
More general nonlinearities are considered in \cite{MP} for the zero and the positive mass case, 
and the existence of at least one radial solutions is established through the mountain pass theorem and a similar approximation in \cite{BoDeCoDe12}. 
Finally, \cite{BIMM} develops the monotonicity trick due to \cite{St88,JeTo98,Je99} for nonsmooth functionals in the setting of \cite{Sz86} and 
the existence of infinitely many radial solutions of \eqref{eq:BI} is proved. 
In addition, the existence of nonradial solutions is also shown when $N = 4$ or $N \geq 6$. 
We also refer to \cite{BdAP} for the regularity result of radial local minimizers 
and \cite{BMP} for the $L^2$-normalized solutions. 
It is worth mentioning that in all these works, radial symmetry (or some symmetry of functions) plays a crucial role 
to overcome the lack of compactness. 
Due to the usage of the symmetry, 
the existence of nonradial solutions to \eqref{eq:BI} when $N=5$ is missing in \cite{BIMM}.

In this paper, we aim to develop a concentration compactness argument or profile decompositions 
for the nonsmooth functional corresponding to \eqref{eq:BI} and this allows us to discuss 
the existence of a ground state solution to \eqref{eq:BI}, i.e., 
a solution with the least possible energy, which seems to be the most relevant from the physical point of view. 
Moreover, the argument can be applied for the existence of nonradial solutions when $N = 5$. 
For this purpose, we introduce a new direct variational approach based
on minimization over a Poho\v{z}aev set, as we shall see later.
Since we study \eqref{eq:BI} in the zero and the positive mass case as in \cite{MP,BIMM}, 
two sets of assumptions are provided. After the existence of ground state solutions, 
nonradial solutions are studied. 
We start our considerations with the \textit{zero mass case} in the spirit of Berestycki and Lions \cite{BeLi83-1}.

\subsection{The zero mass case}

In this subsection, the results for the zero mass case are stated. 
To explain our setting, we first introduce a function space we work on in this case. 
The following choice is taken from \cite{BIMM24,BIMM}. See \cite{BdAP} for another choice of function space.

Let $N \geq 3$ and choose $j_0 \in \N$ such that $\max \{ 2^*, N \} < 2j_0$. Then we introduce the function space
\[
X_0 := \left\{ u \in L^{2^*}(\R^N) \ : \ \|u\|_{X_0} := \| \nabla u \|_2 + \|\nabla u\|_{2 j_0} < \infty \right\}.
\]
Thanks to $N < 2j_0$, it is easily seen from Sobolev's and the interpolation inequality that 
\begin{equation}\label{def:C0}
X_0 \subset W^{1,2^*} (\R^N) \cap \cC_0(\R^N), \quad \cC_0(\R^N) := \left\{ u \in \cC(\R^N) \ : \ \lim_{|x| \to \infty} u(x) = 0 \right\}. 
\end{equation}
Note that if $u \in X_0$ and $\| \nabla u \|_\infty \leq 1$, then the inequality $1 - \sqrt{1 - t^2 } \leq t^2$ in $[0,1]$ gives 
\[
\Psi(u) = \int_{\R^N} \left(1 - \sqrt{1-|\nabla u|^2}\right) \, dx \leq \int_{\R^\N} \left| \nabla u \right|^2 dx < \infty. 
\]
This observation allows us to define the \textit{domain} of $\Psi$ as
\[
D(\Psi) := \left\{ u \in X_0 \ : \ \|\nabla u\|_\infty \leq 1 \right\}.
\]
Then $\Psi : X_0 \rightarrow \R \cup \{\infty\}$ is a convex functional with a convex domain $D(\Psi)$.

Next, to treat $\Phi$ in \eqref{eq:I} on $X_0$, we introduce the following assumptions on $f$: 
\begin{enumerate}[label=(F\arabic{*}),ref=F\arabic{*}]\setcounter{enumi}{0}
\item \label{F1} $f \in \cC^1(\R)$ and $f(0)=0$;
\item \label{F2} $\displaystyle \limsup_{s \to 0} \frac{f'(s)}{|s|^{2^*-2}} = 0$;
\item \label{F3} $sH'(s) > 0$ for all $s \neq 0$ where $H(s) := F(s) + \frac{1}{N} f(s)s$;
\item \label{F4} $\displaystyle \lim_{|s|\to\infty} H(s) = \infty$;
\item \label{F5} $s \mapsto f(s)s$ is nonincreasing in $(-\infty, 0)$ and nondecreasing in $(0,\infty)$;
\item \label{F6} $|F(s)| \lesssim |f(s)s|$ for $s \in \R$, that is, there exists a $C \geq 0$ such that $|F(s)| \leq C |f(s)s|$ for every $s \in \R$. 
\end{enumerate}

\begin{Rem}
	\begin{enumerate}[label=(\roman*)]
		\item 
		Under \eqref{F1}, condition \eqref{F5} is equivalent to $ 0 \leq f'(s) s^2 + f(s)s$ for every $s \in \R$. 
		Thus, a direct computation shows that \eqref{F1} and \eqref{F5} with an additional condition $sf(s) > 0$ for each $s \neq 0$ imply \eqref{F3}. 
		\item 
		Condition \eqref{F6} cannot be derived from \eqref{F1}--\eqref{F5} and such an example of $f$ is given as follows. 
		Let $k_0 \in \N$ satisfy $2^* < 2k_0$ and consider 
		$F(s) := \log ( 1 + s^{ 2k_0 }) \in \cC^\infty(\R)$. Since $f(s) = 2k_0 s^{2k_0-1} / (1 + s^{2k_0})$, 
		it is immediate to verify \eqref{F1}--\eqref{F5}, however, \eqref{F6} is not satisfied.

	\end{enumerate}

\end{Rem}

By \eqref{def:C0}, \eqref{F1} and \eqref{F2}, we may verify $\Phi \in \cC^1(X_0;\R)$ and 
it allows us to define the notion of a critical point of $I$ as in \cite{Sz86}. Namely, $u \in X_0$ is said to be a \textit{critical point of $I$} if
\begin{equation}\label{def-cri}
	\Psi(v)-\Psi(u) - \Phi'(u)(v-u) \geq 0 \quad \text{for all $v \in X_0$}, 
\end{equation}
that is, $\Phi'(u) \in \partial \Psi(u)$, where $\partial \Psi(u)$ denotes the \textit{subdifferential} of a convex functional $\Psi$.

To discuss the existence of ground state solutions of \eqref{eq:BI}, 
we will introduce the Poho\v{z}aev set $\cM_0$ for \eqref{eq:BI}. Let
\begin{equation}\label{eq-def-funcM0}
	M_0(u) := \Psi(u) - \int_{\R^N} H(u) \, dx = \Psi(u) - \Phi(u) - \frac{1}{N} \Phi'(u)(u)
\end{equation}
and define $\cM_0$ by 
\begin{equation}\label{eq-def-setM0}
	\cM_0 := \left\{ u \in X_0 \setminus \{0\} \ : \ M_0(u) = 0 \right\}.
\end{equation}
Since for $u \in X_0 \setminus D(\Psi)$ 
$\Psi(u) = \infty$ and $M_0(u) = \infty$ hold, we easily get $\cM_0 \subset D(\Psi) \setminus \{0\}$. 
For the semilinear scalar field problems, it is known that each reasonable solution satisfies the Poho\v{z}aev identity (for instance, see \cite[Proposition 1]{BeLi83-1}), 
and hence any reasonable solution belongs to the Poho\v{z}aev set. 
A counterpart for \eqref{eq:BI} of this fact is obtained in \cite[Proposition 2.6]{BIMM} (see also \cref{l:RegCritP}). 
Therefore, we are led to the following minimizing problem: 
\[
c_0 := \inf_{\cM_0} I. 
\]
Remark that if there exists a minimizer $u$ corresponding to $c_0$ and if $u$ is a nontrivial solution of \eqref{eq:BI}, 
then $u$ has the least energy among all nontrivial solutions of \eqref{eq:BI} in $X_0$, and hence $u$ is a ground state solution of \eqref{eq:BI}.

\begin{Th}\label{th:1}
Let $N \geq 3$ and assume  \eqref{F1}--\eqref{F6}. Then the following assertions hold: 
\begin{enumerate}[label={\rm (\roman*)}]
	\item 
	$0<c_0 = c_{0,MP}$ holds where $c_{0,MP}$ is a mountain pass level given by \eqref{eq:MP};
	
	\item 
	every minimizing sequence $(u_n)_{n \in \N}$ for $c_0$ is relatively compact in $X_0$ up to translations,  
	and if $\| u_n - u_0 \|_{X_0} \to 0$, then $u_0 \in \cM_0$ attains $c_0$ and is a ground state solution of \eqref{eq:BI}, 
	namely $u_0$ is a classical solution of \eqref{eq:BI} with $I(u_0) = c_0$;
	
	\item 
	for each ground state solution $u$ of \eqref{eq:BI}, either $u>0$ in $\R^N$ or $u<0$ in $\R^N$ holds, and 
	there exists a ground state solution of \eqref{eq:BI} which is radially symmetric. 
\end{enumerate}
\end{Th}

\begin{Rem}\label{Rem:ground}
	By \cref{th:1} (iii), the ground state energy of \eqref{eq:BI} is equal to 
	the \emph{radial} ground state energy, that is, the least energy among all \emph{radial} nontrivial solutions of \eqref{eq:BI} in $X_0$. 
	However, a priori, it is not clear whether the ground state energy coincides with the radial one 
	since we do not impose on $f$ the oddness. 
\end{Rem}

When $f(u)=|u|^{p-2} u$ with $p > 2^*$, 
as in the seminlinear case (see \cite{MedJDE}), the existence of minimizers corresponding to $c_0$ 
implies the Sobolev type inequality \eqref{Sob-type-ineq} obtained in \cite{BoDeCoDe12}, 
the expression of the best constant with $c_0$ 
and the characterization of every optimizer with the ground state solutions of \eqref{eq:BI}:

\begin{Th}\label{th:2}
Let $N \geq 3$ and $p > 2^*$. Then for any $u \in X_0$ there holds
\begin{equation}\label{ineq:Sobtyp}
\int_{\R^N} \left(1 - \sqrt{1-|\nabla u|^2}\right) \, dx \geq \left( \frac{N+p}{Np} \right) N^{\frac{p}{N+p}} \left( \inf_{\cM_0} I \right)^{\frac{p}{N+p}} \left( \int_{\R^N} |u|^p \, dx \right)^{\frac{N}{N+p}}.
\end{equation}
Moreover, every optimizer $v$ has a form of $v(x) = \theta u(x/\theta)$, 
where $\theta>0$ and $u$ is a ground state solution of \eqref{eq:BI} with $f(u) = |u|^{p-2}u$.
\end{Th}

\begin{Rem}
	\begin{enumerate}
		\item 
		A difference between \cite[Theorem 1.2]{BoDeCoDe12} and \cref{th:2} 
		is in the characterization of optimizer through ground state solutions of \eqref{eq:BI} with $f(u) = |u|^{p-2}u$.
		In fact, in \cite[Theorem 1.2]{BoDeCoDe12} it is shown that the constant is achieved by a radial solution of \eqref{eq:BI}. 
		
		\item 
		When $p=2^*$, we may also obtain a similar inequality to \eqref{ineq:Sobtyp} with the best constant involving $\inf_{\cM_0} I$. 
		However, in this case, the constant is never achieved as shown in \cite{BoDeCoDe12}. 
		
	\end{enumerate}
\end{Rem}

\subsection{Positive mass case}

In this case, let $N \geq 1$ 
and suppose that the nonlinearity $f$ is of the form $f(s) = -s + g(s)$, where $g \in \cC^1 (\R)$ satisfies the following conditions:
\begin{enumerate}[label=(G\arabic{*}),ref=G\arabic{*}]\setcounter{enumi}{0}
\item \label{G1} $g'(0) = 0 = g(0)$;
\item \label{G2} the map $\displaystyle s \mapsto \frac{g(s)}{|s|}$ is strictly increasing on $\R$;
\item \label{G3} $\displaystyle \lim_{|s|\to\infty} \frac{g(s)}{s} = \infty$.
\end{enumerate}
In this case, choose $j_0$ such that $2 j_0 > N$ and introduce the space
\[
X_1 := \{ u \in H^1 (\R^N) \ : \ \nabla u \in L^{2j_0} (\R^N) \}
\]
with the norm
\[
\|u\|_{X_1} := \|u\|_{H^1} + \|\nabla u\|_{2j_0}.
\]
Then, the variational functional $I$ is given by \eqref{eq:I}, $\Psi$ is defined as before with the natural domain
\[
D(\Psi) := \left\{ u \in X_1 \ : \ \|\nabla u\|_\infty \leq 1 \right\}
\]
and 
\[
\Phi(u) :=  \int_{\R^N} - \frac12 u^2 + G(u) \, dx, \quad G(s) := \int_0^s g(t) \, dt.
\]
Remark that $\Psi$ is a convex functional on a convex domain $D(\Psi)$ and $\Phi \in \cC^1 (X_1; \R)$. 
Therefore, the notation of critical points of $I$ on $X_1$ is defined as in \eqref{def-cri} by changing $X_0$ to $X_1$. 
We also introduce the Poho\v{z}aev set as
\begin{equation*}\label{eq-def-M1}
\cM_1 := \left\{ u \in X_1 \setminus \{0\} \ : \ M_1(u) = 0 \right\} \subset D(\Psi) \setminus \{0\}, \quad M_1(u) := \Psi(u) - \Phi(u) - \frac{1}{N} \Phi'(u)(u)
\end{equation*}
and consider the following minimizing problem: 
\[
c_1 := \inf_{\cM_1} I.
\]
Then $c_1$ gives the ground state energy as follows: 

\begin{Th}\label{th:3}
Suppose $N \geq 1$, $f(s) = -s +g(s)$ and \eqref{G1}--\eqref{G3}. Then 
\begin{enumerate}[label={\rm (\roman *)}]
	\item 
	$0<c_1 = c_{1,MP}$ holds where $c_{1,MP}$ is a mountain pass level; 
	
	\item 
	every minimizing sequence $(u_n)_{n \in \N}$ for $c_0$ is relatively compact in $X_1$ up to translations, 
	and if $\| u_n - u_0 \|_{X_1} \to 0$, then $u_1 \in \cM_1$ attains $c_1$ and is a ground state solution of \eqref{eq:BI}, 
	namely $u_0$ is a classical solution of \eqref{eq:BI} with $I(u_0) = c_1$;
	
	\item 
	for each ground state solution $u$ of \eqref{eq:BI}, either $u>0$ in $\R^N$ or $u<0$ in $\R^N$ holds, 
	and there exists a ground state solution of \eqref{eq:BI} which is radially symmetric. 
	\end{enumerate}
\end{Th}

\subsection{Nonradial solutions}

Now we show that \eqref{eq:BI} also admits nonradial solutions provided $N \geq 4$. 
Let $X$ denote the space $X_0$ in the {\em zero-mass case}, or $X_1$ in the {\em positive-mass case}. Similarly, let $\cM$ denote the corresponding Poho\v{z}aev set, $\cM_0$ or $\cM_1$, respectively. 
We treat both cases simultaneously and assume either \eqref{F1}--\eqref{F6} or \eqref{G1}--\eqref{G3}, depending on which case is considered.

The approach here is taken from \cite{BartschWillem,BiSi21,JeLu20,Med20,MedJDE,Wi96}. 
Let $k_1,k_2 \in \N$ satisfy $k_1,k_2 \geq 2$ and $N-k_1-k_2 \geq 0$ 
and write $x=(x_1,x_2,x_3) \in \R^{k_1} \times \R^{k_2} \times \R^{N-k_1-k_2}$ 
(if $N=k_1+k_2$, then we write $x=(x_1,x_2) \in \R^{k_1} \times \R^{k_2}$). 
Define 
\[
\begin{aligned}
	\cO &:= \cO(k_1) \times \cO(k_2) \times \{ \mathrm{id} \} \subset \cO(N),
	\\
	X_{\cO} &:= \left\{ u \in X \ : \ u(\rho x) = u(x) \quad \text{for any $x \in \R^N$ and $\rho \in \cO$} \right\}.
\end{aligned}
\]
For each $u \in X_\cO$, we identity $u(x_1,x_2,x_3)$ with $u(r_1,r_2,x_3)$ where $r_1 = |x_1|$ and $r_2 = |x_2|$. 
Then we define $\tau$ by 
\[
\tau(r_1,r_2,x_3) := (r_2,r_1,x_3), \quad 
(\tau u) (x_1,x_2,x_3) := - u( \tau( |x_1|,|x_2|,x_3 ) ) = -u( |x_2|,|x_1|,x_3 ). 
\]
Remark that $\tau : X_\cO\to X_\cO$ is an isometry. We also set 
\[
X_\tau := \left\{ u \in X_\cO \ : \ \tau u = u  \right\}. 
\]
Clearly, if $u \in X_\tau$ is radial, i.e. $u(x) = u(\rho x)$ for any $\rho \in \cO(N)$, then $u \equiv 0$. 
Hence $X_\tau$ does not contain any nontrivial radial functions.

\begin{Th}\label{th:4} 
	Let $N \geq 4$ and fix $k_1,k_2 \in \N$ so that $k_1,k_2 \geq 2$ and $N-k_1-k_2 \geq 0$. 
	Suppose \eqref{F1}--\eqref{F6} (resp. \eqref{G1}--\eqref{G3}) and that $f$ (reps. $g$) is odd. 
	Then there is a nontrivial critical point $u_0\in \cM\cap X_\tau =:\cM_\tau$ such that 
	$I(u_0) = \inf_{\cM_\tau} I \geq \inf_{\cM} I > 0$. Moreover, when $k_1=k_2$, then 
	\begin{equation*}\label{eq:thmain4}
	I(u_0) = \inf_{\cM_\tau }I > 2\inf_{\cM} I > 0.
	\end{equation*}
In particular $u_0$ is a nonradial classical solution of \eqref{eq:BI}. 
\end{Th}

\subsection{Comments}

We describe arguments to prove \cref{th:1,th:3} and their difficulties. 
As explained above, in the literature, the existence of ground state solutions of \eqref{eq:BI} with a general nonlinearity $f$ is not studied. 
As pointed as in \cref{Rem:ground}, it is not clear that radial functions provide the ground state energy of \eqref{eq:BI} a priori, 
and hence the radial symmetry may not be exploited to prove the existence of ground state solutions for a general nonlinearity $f$. 
Thus, we face the lack of compact embedding of $X_0$ or $X_1$, and 
to the best of the authors' knowledge, this paper is the first to deal with the lack of compactness for \eqref{eq:BI}. 
In addition to the lack of compact embedding, the nonsmoothness of functionals $I,M_0,M_1$ also makes the analysis complicated and different from 
the semilinear case (for instance, \cite{BeLi83-1,Med20,MedJDE}). 
Thus, to solve difficulties and to prove \cref{th:1,th:3}, new ideas or arguments are necessary.

To show the existence of minimizers corresponding to $c_0$ and $c_1$, 
consider any minimizing sequence $(u_n)_{n \in \N} \subset X_0$ for $c_0$ (resp. $(u_n)_{n \in \N} \subset X_1$ for $c_1$). 
Since the Poho\v{z}aev functionals $M_0$ and $M_1$ are different from the semilinear case, 
the boundedness of $(u_n)_{n \in \N}$ in $X_0$ (resp. $X_1$) is an issue. 
To overcome this point, in the zero mass case, \eqref{F6} is used. 
On the other hand, in the positive mass case, the case $N \geq 3$ is simple and similar to the semilinear case, 
however, the cases $N=1,2$ are different. See the beginning of the proof of \cref{th:3}.

After establishing the boundedness of $(u_n)_{n \in \N}$, the next issue is to prove 
the relative compactness of $(u_n)_{n \in \N}$ up to translations. 
Here we employ and develop profile decompositions suited for the nonsmooth functionals $M_0,M_1$ and smooth functional $\Phi$, 
and our argument is inspired by \cite{MedJDE,Gerard}.
For the existence of minimizers, the projection to the Poho\v{z}aev sets $\cM_0$ and $\cM_1$ are constructed and 
we compute and compare energies of projected profiles of $(u_n)_{n \in \N}$ with $c_0$ and $c_1$. 
This detailed analysis also allows us to prove that every minimizer has a constant sign, 
and this fact leads to the existence of radial minimizers.

Even though we know the existence of minimizers, 
it is worth noting that whether these minimizers become critical points of $I$ is not trivial or immediate, 
and this point is more delicate than the semilinear case. 
For the semilinear case, a similar assertion is proved in \cite[Lemma 3.3]{MedJDE} 
where the differentiability of functional originated from the differential operator is utilized. 
In our case, the corresponding functional is $\Psi$, which is not smooth on $X_0$ or $X_1$. 
In addition, the nonsmoothness of $M_0$ and $M_1$ affect the properties of the projections to $\cM_0$ and $\cM_1$, 
and the analysis of the projections onto $\cM_0$, $\cM_1$ and functional $I$ is more delicate and complicated. 
These issues will be solved in \cref{lem:minimizers-are-critical,lem:positive-mass-minimizers-crit} and 
it is one of novelties in this paper.

We emphasize that since the method in this paper does not rely on the compactness of the embedding of function spaces, 
our argument can be applied to obtain the existence of nonradial solutions to \eqref{eq:BI}. 
In particular, \cref{th:4} handles the case $N = 5$ and extends the results in \cite{BIMM}. 
Our setting is slightly different from \cite{BartschWillem,BiSi21,JeLu20,Med20,MedJDE,Wi96}, 
and any combination of $k_1,k_2$ with $k_1,k_2 \geq 2$ and $N-k_1-k_2 \geq 0$ can be treated 
and we obtain other type of nonradial solutions which are not treated in \cite{BIMM}.


It is also worth mentioning that our argument may be extended to a more general problem, 
for instance $f(x,u)$ in place of $f(u)$ in \eqref{eq:BI}. 
In fact, for the semilinear case, 
the characterization of the ground state energy by the mountain pass level is obtained in \cite{JeTa03}, 
and it is useful to handle the nonlinearity $f(x,u)$ or a potential term $V(x) u$. 
The same characterization for \eqref{eq:BI} is obtained in \cref{th:1,th:3}, and we expect that 
this might be helpful to study \eqref{eq:BI} with $f(x,u)$ or $V(x) u$. 
Furthermore, though it is only for minimizers, we succeed to prove that the constraint $M_0(u) = 0$ (resp. $M_1(u) = 0$) is \emph{natural} 
in the sense that a \emph{critical point} of $I |_{\cM_0}$ (resp. $I |_{\cM_1}$) provides a critical point of $I$. 
Therefore, it is interesting to see whether the same is true for higher energies or other nonsmooth constraints.

The structure of the paper is as follows. 
In \cref{Sec:Pre}, the functional setting is recalled and some preliminary properties related to the functionals $I$ and $\Psi$ are proved. 
In \cref{Sec:ProDec} we develop profile decompositions for bounded sequences in $X_0$ and $X_1$. 
\cref{Sec:Zeromass} is then devoted to the zero mass case and contains the proofs of \cref{th:1,th:2}. 
In \cref{sec:positive} we turn to the positive mass case and prove \cref{th:3}. 
Finally, in \cref{sec:nonradial} we establish the existence of nonradial solutions.

\section{Concentration-compactness and a profile decomposition}
\label{Sec:Pre}

In this section, preliminary results are stated and proved in the zero mass and positive mass cases. 
Throughout this section and \cref{Sec:ProDec}, we assume the following: 
\begin{equation}\label{eq-defX}
	p_0 := \begin{dcases}
		2^* & \text{(the zero mass case)},
		\\
		2 & \text{(the positive mass case)},
	\end{dcases}
	\quad 
	X := \begin{dcases}
		X_0 & \text{(the zero mass case)},
		\\
		X_1 & \text{(the positive mass case)}
	\end{dcases}
\end{equation}
and set 
\[
\| u \|_X := \begin{dcases}
	\| u \|_{X_0}  &\text{(the zero mass case)},
	\\
	\| u \|_{X_1} & \text{(the positive mass case)}.
\end{dcases}
\]
Then $(X,\| \cdot \|_X)$ is a reflexive Banach space (cf. \cite[Proposition 3.3]{BIMM24}) and the following embedding holds 
(see \eqref{def:C0} for the definition of $\cC_0(\R^N)$): 
\begin{equation}\label{eq-X-emb}
	X \subset W^{1,p_0} (\R^N) \cap \cC_0(\R^N). 
\end{equation}
Remark also that $D(\Psi) = \{ u \in X \ : \ \| \nabla u \|_\infty \leq 1 \}$ is convex in $X$. 
Similarly, put
\[
M(u) := \begin{dcases}
	M_0(u) & \text{(the zero mass case)}, \\
	M_1(u) & \text{(the positive mass case)}. 
\end{dcases}
\]
Finally, 
the space dimension $N$ is assumed to be $N \geq 3$ in the zero mass case and $N \geq 1$ in the positive mass case.

We first prove that the functional $\Psi$ restricted on $D(\Psi)$ is continuous:

\begin{Lem}\label{Lem-conti-Psi}
	Assume $(u_n)_{n \in \N} \subset D(\Psi)$, $u \in D(\Psi)$ and $\| u_n - u \|_X \to 0$ as $n \to \infty$. 
	Then $\Psi(u_n) \to \Psi(u)$ as $n \to \infty$. 
\end{Lem}

\begin{proof}
Notice that 
\begin{equation}\label{ineq-t^2}
	\frac{1}{2} t^2 \leq \sqrt{1-t^2} - 1 \leq t^2 \quad \text{for any $t \in [-1,1]$}. 
\end{equation}
Remark also that $\| u_n - u \|_X \to 0$ yields $\| \nabla u_n - \nabla u \|_2 \to 0$. 
Then we may find a subsequence $(u_{n_k})_{k \in \N}$ and $w \in L^2(\R^N)$ such that 
$\nabla u_{n_k}(x) \to \nabla u(x)$ a.e. $\R^N$ and $| \nabla u_{n_k}(x) | + | \nabla u(x) | \leq w(x)$ a.e. $\R^N$. 
By \eqref{ineq-t^2} with $|\nabla u(x)| \leq 1$ and $|\nabla u_{n_k}(x)| \leq 1$ due to $u_{n_k},u \in D(\Psi)$, 
the dominated convergence theorem gives $\Psi(u_{n_k}) \to \Psi(u)$ as $k \to \infty$. 
Since the limit is independent of subsequences, we have $\Psi(u_n) \to \Psi(u)$ as $n \to \infty$ 
and this completes the proof. 
\end{proof}

Next we introduce 
\[
\Psi_0(u) := 
\begin{dcases}
	\Psi(u) & \text{(the zero mass case)},
	\\
	\Psi(u) + a \| u \|^2_2 & \text{(the positive mass case)},
\end{dcases}
\]
where $a \in (0,\infty)$ is any fixed number.

\begin{Lem}\label{l:strconv}
	Suppose that $(u_n)_{n \in \N} \subset D(\Psi)$ satisfies $u_n \rightharpoonup u$ weakly in $X$ and $\Psi_0(u_n) \to \Psi_0(u)$. 
	Then $\| u_n - u \|_X \to 0$ holds. 
\end{Lem}

\begin{proof}
The following argument is essentially contained in \cite[Proof of Proposition 3.11]{BIMM24}. 
From the uniform convexity of $L^p$ spaces, it suffices to prove 
$\| \nabla u_n \|_{2j} \to \| \nabla u \|_{2j}$ for each $j \in \N$ and $\| u_n \|_2 \to \| u \|_2$ in the positive mass case. 
Since $\Psi$ is weakly lower semicontinuous due to the convexity, we have 
$\Psi(u) \leq \liminf_{n \to \infty} \Psi(u_n)$. 
Moreover, in the positive mass case, 
$\Psi(u_n) \to \Psi(u)$ and $\| u_n \|_2 \to \| u \|_2$ hold. 
In fact, suppose that there exists $(u_{n_k})_{k \in \N}$ such that $\Psi(u) < \limsup_{k \to \infty} \Psi(u_{n_k})$. 
By passing to a subsequence if necessary, we may also suppose $ \limsup_{k \to \infty} \Psi(u_{n_k}) = \lim_{k \to \infty} \Psi(u_{n_k})$. Then 
\[
\begin{aligned}
	\Psi_0(u) = \Psi(u) + a \| u \|_2^2 
	<
	\lim_{k \to \infty} \Psi(u_{n_k}) + a \liminf_{k \to \infty} \| u_{n_k} \|_2^2 
	=
	\liminf_{k \to \infty} \Psi_0(u_{n_k}) = \Psi_0(u),
\end{aligned}
\]
which is a contradiction. Hence $\limsup_{n \to \infty} \Psi(u_n) \leq \Psi(u)$ and 
$\Psi(u_n) \to \Psi(u)$ holds. 
In a similarly way, $\| u_n \|_2 \to \| u \|_2$ can be proved and $\| u_n - u \|_2 \to 0$ holds in the positive mass case.

To prove $\| \nabla u_n \|_{2j} \to \| \nabla u \|_{2j}$ from $\Psi(u_n) \to \Psi(u)$, 
let us recall the following expansion (\cite{K1})
\begin{equation}\label{exp-sq-1t^2}
	1 - \sqrt{1-t^2} = \sum_{j \in \N} b_j t^{2j}, \quad b_j > 0.
\end{equation}
In particular, 
\[
\Psi(u) = \sum_{j \in \N} b_j \| \nabla u \|_{2j}^{2j} \quad \text{for any $u \in D(\Psi)$}. 
\]
For each $k \in \N$, set 
\[
\Psi_{k,1}(u) := \sum_{j=1}^k b_j \| \nabla u \|_{2j}^{2j}, \quad \Psi_{k,2}(u) := \sum_{j=k+1}^\infty b_j \| \nabla u \|_{2j}^{2j}.
\]
It is not hard to see that $\Psi_{k,1}$ and $\Psi_{k,2}$ are convex on $D(\Psi)$, and weakly lower semicontinuous. 
By utilizing the above argument and $\Psi(u_n) \to \Psi(u)$, 
we obtain $\Psi_{k,1} (u_n) \to \Psi(u)$ and $\Psi_{k,2}(u_n) \to \Psi(u)$ for each $k \in \N$. 
By starting from the case $k=1$, $\| \nabla u_n \|_{2j}^{2j} \to \| \nabla u \|_{2j}^{2j}$ is obtained for every $j \geq 1$. 
This completes the proof. 
\end{proof}

Finally, we prove that every critical point of $I$ in $X$ becomes a classical solution of \eqref{eq:BI} and satisfies the Poho\v{z}aev identity. 
A similar assertion is obtained in \cite[Propositions 2.6 and 2.10]{BIMM} for $N \geq 3$. 
Here our main focus is in the case $N=1$. Notice that under \eqref{F1}--\eqref{F2} (resp. \eqref{G1}), $\Phi \in \cC^1(X;\R)$ due to \eqref{eq-X-emb}.

\begin{Lem}\label{l:RegCritP}
	Suppose either $N \geq 3$ and \eqref{F1}--\eqref{F2} or else $N \geq 1$ and \eqref{G1}. 
	Then each critical point $u$ of $I$ in $X$ is a classical solution of \eqref{eq:BI} and the Poho\v{z}aev identity $M(u) = 0$ holds. 
\end{Lem}

\begin{proof}
Let $u \in X$ be a critical point of $I$ and define 
\[
\rho(x) := \begin{dcases}
	f(u(x)) & \text{(the zero mass case)}, \\
	-u(x) + g(u(x)) & \text{(the positive mass case)}.
\end{dcases}
\]
By \eqref{eq-X-emb}, $\rho \in L^\infty(\R^N) \cap \cC(\R^N)$ holds and by rewriting \eqref{def-cri}, $u$ is a (unique) minimizer of 
\[
J(v) := \Psi(v) - \int_{\R} \rho v \, dx : X \to \R \cup \{\infty\}.
\]

When $N \geq 3$, from \cite[Propositions 2.13]{BIMM} (cf. \cite[Proposition 2.10]{BIMM}), 
we see that $u$ is a weak solution of \eqref{eq:BI} and $ \| \nabla u \|_\infty < 1$ holds. 
In particular, by \cite{BI2} (cf. \cite[Proof of Theorem 1.11 Step 5]{BIMM24}), $u \in \cC^{1,\alpha} (\R^N)$ for some $\alpha \in (0,1)$ and $u$ is a strong solution of 
\begin{equation}\label{UniEeq}
- \sum_{i,j=1}^N a_{ij} (x) \partial_i \partial_j u = \rho \quad \text{in} \ \R^N, \quad 
a_{ij} (x) := \frac{\delta_{ij}}{\sqrt{ 1 - |\nabla u(x)|^2 }}  + \frac{\partial_i u (x) \partial_j u(x)}{\left( 1 - | \nabla u(x)|^2 \right)^{3/2}}
\end{equation}
Since $a_{ij} \in \cC^{0,\alpha}_{\rm loc}(\R^N)$ and $\rho \in \cC^1(\R^N)$, the Schauder theory gives $u \in \cC^2(\R^N)$.

We remark that the argument in \cite{BIMM} is also valid for $N=2$ in the positive mass case. 
Hence, in this case, $u \in \cC^{1,\alpha} (\R^2)$ is a strong solution of \eqref{UniEeq} and 
the argument above leads to $u \in \cC^2(\R^2)$.

Let $N=1$ and $u \in X$ be a critical point of $I$, 
and notice that $\rho \in L^2(\R) \cap L^\infty(\R)$ in view of \eqref{G1} and \eqref{eq-X-emb}. 
We follow the argument in \cite[Theorem 1.4, Proposition 2.7, Remark 2.9]{BdAP}. 
Without any changes, the argument in \cite[Proposition 2.7]{BdAP} shows that 
\begin{equation}\label{e:inte}
	\frac{(u')^2}{\sqrt{1-(u')^2}} \in L^1(\R), \quad 
	\frac{ u' \psi' }{\sqrt{1-(u')^2}} \in L^1(\R) \quad \text{for each $\psi \in X_1$ with $\psi' \in L^\infty(\R)$}. 
\end{equation}
We next prove that $u$ is a weak solution of \eqref{eq:BI} by a slightly modified argument in \cite[Proof of Theorem 1.4]{BdAP}. 
Let $\psi \in  \cC^\infty_c(\R)$ and $E_k := \{ x \in \R \ | \ |u'(x)| \geq 1 - 1/k  \}$ for $k \in \N$.

If $|E_k| = 0$ holds for some $k \in \N$, 
then $\| u' \|_\infty < 1$ and $u + t \psi \in D(\Psi)$ provided $|t| \ll 1$. 
Since $u$ is a minimizer of $J$, $J(u) \leq J( u + t \psi )$ holds for any $t \in \R$ with $|t| \ll 1$. 
This together with $\| u' \|_\infty < 1$ and the dominated convergence theorem implies 
\[
\int_{\R} \frac{u' \psi'}{\sqrt{1 - (u')^2}} \, dx = \int_\R \rho \psi \, dx = \int_{\R} f(u) \psi \, dx \quad 
\text{for all $\psi \in \cC^\infty_c(\R)$}. 
\]
In particular, $u$ is a weak solution of \eqref{eq:BI}. 
Furthermore, by $\| u' \|_\infty < 1$, the above equality asserts that 
$u' / \sqrt{1 - (u')^2} \in W^{1,\infty} (\R)$ with the derivative $\rho \in \cC^1(\R)$. 
Thus, $u \in \cC^2(\R)$.

In what follows, we may suppose $|E_k| > 0$ for each $k \in \N$. 
Notice that $E_{k+1} \subset E_k$, $|E_2| < \infty$ and $| E_k | \to 0$ as $k \to \infty$ by \eqref{e:inte}. 
Let $n_0 \in \N$ be large enough so that $ 0 < | [-n_0,n_0] \cap E_{2}^c|$ and set  
\[
\eta_0(x) := \frac{1}{ | [-n_0,n_0] \cap E_{2}^c|  } \int_{-\infty}^x \bm{1}_{ [-n_0,n_0] \cap E_{2}^c } (s) \, ds \in  W^{1,\infty} (\R).
\]
For $k \geq 2$, introduce 
\[
a_k := \int_\R \psi'(s) \bm{1}_{E_k^c} (s) \, ds = \int_{E_k^c} \psi'(s) \, ds \in \R. 
\]
Since $|E_k| \to 0$ as $k \to \infty$, the fact $\psi \in \cC^\infty_c(\R)$ yields 
\[
\lim_{k \to \infty} a_k = \int_{\R} \psi' \, ds = 0. 
\]
Define $\psi_k$ by 
\[
\psi_k(x) := \int_{-\infty}^x \psi'(s) \bm{1}_{E_k^c} (s) \, ds - a_k \eta_0 (x). 
\]
Let $R > n_0$ satisfy $\supp \psi \subset [-R,R]$. Since $\eta_0$ and $\psi_k$ are constant on $\R \setminus [-R,R]$, 
the choice of $a_k$ yields $\supp \psi_k \subset [-R,R]$ for every $k \geq 2$ and 
$|\psi_k'(x)| > 0$ implies $x \in E_k^c$. Furthermore, from $a_k \to 0$ and $|E_k| \to 0$ with $\supp \psi_k, \supp \psi \subset [-R,R]$, 
it follows that $\| \psi_k - \psi \|_1 \to 0$. 
By $\| u'+t\psi_k' \|_\infty \leq 1 - 1/(2k)$ for $|t| \ll 1$, 
the dominated convergence theorem yields 
\[
\begin{aligned}
	\lim_{t \to 0} \frac{J(u+t\psi_k) - J(u) }{t} 
	&= 
	\int_{\R} \frac{ u' \psi_k' }{\sqrt{ 1 - (u')^2 }} \, dx - \int_{\R} \rho \psi_k \, dx 
	\\
	&= 
	\int_{\R} \frac{ u' \left( \psi' \bm{1}_{E_k^c}  - a_k \eta_0'  \right)  }{\sqrt{1-(u')^2}} \, dx 
	- \int_{\R} \rho \psi_k \, dx. 
\end{aligned}
\]
Recalling that $u$ is a minimizer of $J$, we observe that for each $k \geq 2$ and $\psi \in \cC^\infty_c(\R)$, 
\[
0 \leq \lim_{t \to 0} \frac{J(u+t\psi_k) - J(u) }{t} 
= \int_{\R} \frac{ u' \left( \psi' \bm{1}_{E_k^c} - a_k \eta_0' \right)  }{\sqrt{1-(u')^2}} \, dx 
- \int_{\R} \rho \psi_k \, dx. 
\]
Letting $k \to \infty$ together with \eqref{e:inte} and $\| \psi_k - \psi \|_1 \to 0$ leads to 
\[
0 \leq \int_{\R} \frac{u'\psi'}{\sqrt{1 - (u')^2}} \, dx - \int_{\R} \rho \psi \, dx \quad \text{for every $\psi \in \cC^\infty_c (\R)$}. 
\]
Since $\psi \in \cC^\infty_c (\R)$ is arbitrary, the above inequality becomes the equality and $u$ is a weak solution of \eqref{eq:BI}, 
$u'/\sqrt{1 - (u')^2} \in W^{1,1}_{\rm loc} (\R)$ and $( u' / \sqrt{1 - (u')^2} )' = - \rho \in \cC^1(\R)$. 
Therefore, $u \in \cC^2(\R)$.

Once we know the regularity of $u$, it is not difficult to verify the Poho\v{z}aev identity. 
For the details, see \cite[Proof of Proposition 2.6]{BIMM}. 
\end{proof}

\section{Concentration-compactness and a profile decomposition}
\label{Sec:ProDec}

In this section, we establish a profile decomposition for bounded sequences in $X$, which is useful to prove the existence of minimizers of $I$ on $\cM_i$. 
Recall that in this section, the same notation to \cref{Sec:Pre} is used: $p_0$, $X$ and $\Psi_0$. 
For the profile decomposition, several lemmas are necessary. Firstly, a slight extension of \cite[Lemma 2.27]{BIMM} is needed:

\begin{Lem}\label{lem:lions}
Let $Q := [0,1]^N$, $(u_n)_{n \in \N} \subset X $ be bounded in $X$ and $\zeta \in \cC( \R ; [0,\infty)  )$ satisfy 
\begin{equation}\label{ass-zeta}
\limsup_{s \to 0} \frac{\zeta(s)}{|s|^{p_0}} = 0.
\end{equation}
Then for every $q \in (p_0,p^*_0)$ and $\varepsilon \in (0,1)$ where $1/p_0^* = (1/p_0 -1/N)_+$, 
there exists $C=C(L, \zeta,q,\varepsilon) > 0$, where $L := \sup_{n \in \N} \| u_n \|_X$, 
such that 
\[
\int_{\R^N} \left| \zeta(u_n) \right| dx \leq \varepsilon + C \left( \sup_{y \in \Z^N} \int_{y+Q} \left| u_n \right|^{q} dx \right)^{ (q-p_0)/q  }. 
\]
 In particular, if 
\[
\lim_{n\to\infty} \sup_{y \in \Z^N} \int_{y+Q} |u_n|^{q} \, dx = 0,
\]
then 
\[\int_{\R^N} \zeta(u_n) \, dx \to 0.\]
\end{Lem}

\begin{proof}
Let $q \in (p_0,p_0^*)$ and $\varepsilon \in (0,1)$. 
From \eqref{eq-defX} and \eqref{eq-X-emb}, $W^{1,p_0} (\R^N) \subset L^q(\R^N)$ holds.  
Since $(u_n)_{n \in \N}$ is bounded in $X$, it is bounded in $L^\infty(\R^N)$. 
By this bound in $L^\infty(\R^N)$ and \eqref{ass-zeta}, there exists $C = C( L, \zeta,q,\varepsilon)  > 0$ such that
\[
\zeta(u_n(x)) \leq \varepsilon |u_n(x)|^{p_0} + C |u_n(x)|^q \quad \text{for all $x \in \R^N$ and $n \in \N$},
\]
which leads to 
\begin{equation}\label{bdd-zeta-un}
	\int_{\R^N} \zeta(u_n) \, dx \leq \varepsilon \|u_n\|_{p_0}^{p_0} + C \|u_n\|_q^q \quad \text{for all $n \in \N$}.
\end{equation}
Recall $L = \sup_{n \in \N} \| u_n \|_X < \infty$ and 
notice also that $W^{1,p_0} (Q) \subset L^q(Q)$, $X \subset W^{1,p_0} (\R^N) \cap \cC_0(\R^N)$ and 
\[
\begin{aligned}
	\| u_n \|_{q}^q &= \sum_{y \in \Z^N} \| u_n \|_{L^q(y+Q)}^q
	\\
	& \leq \left( \sup_{y \in \Z^N} \| u_n \|_{L^q(y+Q)}^{q-p_0} \right) \sum_{y \in \Z^N} \| u_n \|_{L^q(y+Q)}^{p_0} 
	\\
	& \leq \left( \sup_{y \in \Z^N} \| u_n \|_{L^q(y+Q)}^{q-p_0} \right) \sum_{y \in \Z^N} C_{N,p_0,q} \| u_n \|_{W^{1,p_0} (y+Q)}^{p_0}
	\\
	& = C_{N,p_0,q} \| u_n \|_{W^{1,p_0}(\R^N)}^{p_0} \sup_{y \in \Z^N} \| u_n \|_{L^q(y+Q)}^{q-p_0}
	\\
	& \leq C_{N,p_0,q} C_{L,N,p_0} \sup_{y \in \Z^N} \| u_n \|_{L^q(y+Q)}^{q-p_0}. 
\end{aligned}
\]
Thus \eqref{bdd-zeta-un} implies 
\[
\int_{\R^N} \zeta(u_n) \, dx \leq \left( \sup_{n \geq 1} \| u_n \|_{p_0}^{p_0} \right) \varepsilon + C \sup_{y \in \Z^N} \left( \int_{y+Q} \left|u_n\right|^q dx \right)^{(q-p_0)/q}.
\]
Thus we complete the proof. 
\end{proof}

%

\begin{Lem}\label{lem:BL}
Suppose that $\zeta : \R \rightarrow [0,\infty)$ is of class $\cC^1$ such that $|\zeta'(s)| \lesssim |s|^{p_0-1}$ for $|s| \leq 1$ and $\zeta(0)=0$. 
Let $(u_n) \subset X $ be a bounded sequence such that $u_n(x) \to u_0(x)$ a.e. $x \in \R^N$. Then
\[
\lim_{n\to\infty} \left( \int_{\R^N} \zeta(u_n) \, dx - \int_{\R^N} \zeta(u_n - u_0) \, dx \right) = \int_{\R^N} \zeta(u_0) \, dx.
\]
\end{Lem}

\begin{proof}
Observe that
\begin{align*}
\int_{\R^N} \zeta(u_n) \, dx - \int_{\R^N} \zeta(u_n - u_0) \, dx &= - \int_{\R^N} \int_0^1  \frac{d}{ds} \zeta(u_n - s u_0) \, ds \, dx = \int_0^1 \int_{\R^N}  \zeta'(u_n - s u_0) u_0 \, dx \, ds \\
&\to \int_0^1 \int_{\R^N}  \zeta'(u_0 - s u_0) u_0 \, dx \, ds = - \int_{\R^N} \int_0^1  \frac{d}{ds} \zeta(u_0 - s u_0) \, ds \, dx \\
&= \int_{\R^N} \zeta(u_0) \, dx - \int_{\R^N} \zeta(u_0 - u_0) \, dx = \int_{\R^N} \zeta(u_0) \, dx,
\end{align*}
where to pass to the limit we used Vitali's convergence theorem, $X \subset L^\infty(\R^N)$, the boundedness of $(u_n)_{n \in \N}$ in $X$ and the estimate
$\left| \zeta'(u_n - s u_0) \right| \lesssim |u_n - su_0|^{p_0-1}$ a.e. $\R^N$. 
\end{proof}

Finally we prove the profile decomposition.

\begin{Prop}\label{Prop:decomp}
Let $(u_n)_{n \in \N} \subset X$ be bounded in $X$. 
Then there are $K \in \{0,1,2,\ldots\} \cup \{\infty\}$, $(y_n^j)_{n \in \N} \subset \R^N$, $\widetilde{u}_j \in X$ for $0 \leq j < K+1$ such that 
for any $a \in (0,\infty)$ in the positive mass case, 
\[
\begin{aligned}
u_n(\cdot + y_n^j) &\weakto \widetilde{u}_j & &\text{for $0 \leq j < K +1$}, \\
\widetilde{u}_j &\neq 0  & &\text{for } 1 \leq j < K+1, \\
\lim_{n\to\infty} \Psi_0(u_n) &\geq \sum_{j=0}^K \Psi_0(\widetilde{u}_j)   
\end{aligned}
\]
and
\[
\lim_{n \to \infty} \int_{\R^N} \zeta(u_n) \, dx = \sum_{j=0}^K \int_{\R^N} \zeta(\widetilde{u}_j) \, dx
\]
for any function $\zeta : \R \rightarrow [0,\infty)$ of class $\cC^1$ such that $|\zeta'(s)| \lesssim |s|^{p_0-1}$ for $|s| \leq 1$ and 
\[
\limsup_{s \to 0} \frac{\zeta(s)}{|s|^{p_0}} = 0.
\]
\end{Prop}

\begin{proof}
Fix any $q \in (p_0,p_0^*)$. 
Since $(u_n)_{n \in \N}$ is bounded in $X$, up to a subsequence, we may assume $u_n \rightharpoonup \wt u_0 $ 
weakly in $X$, $u_n \to u_0$ a.e. $\R^N$ and set 
\[
v^0_n := u_n -\wt u_0, \quad 
d_0 := \limsup_{n\to\infty} \sup_{y \in \Z^N} \int_{y+Q} |v^0_n|^{q} \, dx, \quad Q := [0,1]^N.
\]
If $d_0=0$, then we may put $K=0$ and $y_n^0 = 0$. The convexity of $\Psi_0$ and \cref{lem:lions,lem:BL} yield 
\[
	\liminf_{n \to \infty} \Psi_0(u_n) \geq \Psi_0(\wt u_0), \quad 
	\lim_{n \to \infty} \int_{\R^N} \zeta (u_n) \, dx = \int_{\R^N} \zeta (\wt u_0) \, dx 
	+ \lim_{n \to \infty} \int_{\R^N} \zeta (v^0_n ) \, dx 
	= \int_{\R^N} \zeta (\wt u_0) \, dx  . 
\]
Hence the assertions hold.

On the other hand, when $d_0>0$, up to a subsequence, 
we may choose $(y_n^1)_{n \in \N} \subset \R^N$ such that 
\[
\int_{y_n^1 + Q} |v^0_n|^{q} \, dx \to d_0 > 0.
\]
It is easily seen that $|y_n^1| \to \infty$ and we may also suppose 
$v^0_n(\cdot + y_n^1) \rightharpoonup \wt u_1 \not \equiv 0$ weakly in $X$ as well as 
$v^0_n (\cdot + y_n^1) \to \wt u_1$ a.e. on $\R^N$.  
Clearly, 
\[
d_0 = \int_{Q} |\wt u_1|^{q} \, dx. 
\]
We set 
\[
v^1_n := v^0_n - \wt u_1 (\cdot - y^1_n) = u_n - \sum_{j=0}^1 \wt u_j (\cdot - y^j_n), \quad 
d_1 := \limsup_{n \to \infty} \sup_{y \in \Z^N} \int_{ y + Q } |v^1_n|^{q} \, dx. 
\]
Notice that \cref{lem:BL} gives 
\[
\int_{\R^N} \zeta(u_n) \, dx = \int_{\R^N} \zeta ( \wt u_0 ) \, dx + \int_{\R^N} \zeta ( v^0_n ) \, dx  + o(1)
= \int_{\R^N} \zeta ( \wt u_0 ) \, dx + \int_{\R^N} \zeta ( \wt u_1 ) \, dx  
+ \int_{\R^N} \zeta ( v^1_n ) \, dx  + o(1).
\]

When $d_1 = 0$, \cref{lem:lions} implies 
\[
\lim_{n \to \infty} \int_{\R^N} \zeta (u_n) \, dx = \sum_{j=0}^1 \int_{\R^N} \zeta ( \wt u_j ) \, dx. 
\]
For the assertion on $\Psi_0$, let $R>0$ be arbitrary and for a domain $\Omega \subset \R^N$ write 
\[
\Psi_{0,\Omega} (u) := 
\begin{dcases} 
	\int_\Omega 1 - \sqrt{1 - |\nabla u|^2} \, dx & \text{(the zero mass case)},\\
	\int_\Omega 1 - \sqrt{1 - |\nabla u|^2} + a u^2 \, dx & \text{(the positive mass case)}.
\end{dcases}
\]
Since $y^0_n = 0$, $|y_n^1| \to \infty$ and $ u_n(  \cdot + y^j_n) \rightharpoonup \wt u_j$ weakly in $X$, we obtain 
\[
\liminf_{n\to\infty} \Psi_0(u_n) \geq \liminf_{n \to \infty} \sum_{j=0}^1 \Psi_{0, B(y^j_n, R) } (u_n) 
\geq \sum_{j=0}^1 \liminf_{n \to \infty} \Psi_{0,B(0,R)} \left( u_n(\cdot + y_n^j) \right)
\geq \sum_{j=0}^1 \Psi_{0,B(0,R)}  (\wt u_j). 
\]
Letting $R \to \infty$ yields
\[
\liminf_{n \to \infty} \Psi_0(u_n) \geq \sum_{j=0}^1 \Psi_0(\wt u_j)
\]
and we complete the proof by setting $K=1$.

Now let us consider the case $d_1 > 0$. Since $|y_n^1| \to \infty$ and $v_n^1 = v_n^0 - \wt u_1 (\cdot - y_n^1)$, it is not difficult to see 
\[
d_1 \leq d_0. 
\]
After taking a subsequence if necessary, we may choose $(y^2_n)_{n \in \N} \subset \R^N$ so that 
\[
\int_{y^2_n + Q} |v^1_n|^{q}  \, dx \to d_1 > 0.
\]
Then $|y^2_n - y_n^j| \to \infty$ for $j=0,1$ and $v^1_n (\cdot + y^2_j) \rightharpoonup \wt u_2 \not \equiv 0$ weakly in $X$ 
as well as $v^1_n (\cdot + y^2_n) \to \wt u_2$ a.e. $\R^N$. Moreover, $\int_Q \left| \wt u_2 \right|^q dx = d_1$ holds. 
We define 
\[
v_n^2 := v_n^1 - \wt u_2 ( \cdot - y^2_n), \quad d_2 := \limsup_{n \to \infty} \sup_{y \in \Z^N} \int_{y +Q  } |v_n^2|^{q} \, dx. 
\]
\cref{lem:BL} gives 
\[
\int_{\R^N} \zeta ( v_n^1 ) \, dx = \int_{\R^N} \zeta (\wt u_2 ) \, dx + \int_{\R^N} \zeta ( v_n^2 ) \, dx + o(1), \quad 
\int_{\R^N} \zeta (u_n) \, dx = \sum_{j=0}^2 \zeta (\wt u_j) \, dx + \int_{\R^N} \zeta (v_n^2) \, dx + o(1). 
\]
If $d_2=0$, then we set $K=2$ and may prove the assertions similarly in the case $K = 1$.

On the other hand, if $d_2>0$, then by $|y_n^1| \to \infty$, $|y_n^1 - y_n^2| \to \infty$ and $v^2_n = v_n^1 - \wt u_2 (\cdot - y_n^2)$, 
$d_2 \leq d_1$ holds. Furthermore, we find $(y^3_n)_{n \in \N}$ and $\wt u_3 \not \equiv 0$ such that 
\[
v^2_n (\cdot + y_n^3) \rightharpoonup \wt u_3 \not \equiv 0, \quad 
d_2 = \lim_{n \to \infty} \int_{y^3_n + Q} |v^2_n|^{q} \, dx = \int_{Q} | \wt u_3 |^{q} \, dx, \quad |y_n^3-y_n^j| \to\infty \ (j=0,1,2). 
\]
By setting 
\[
v^3_n := v^2_n - \wt u_3 (\cdot - y^3_n), \quad d_3 := \limsup_{n \to \infty} \sup_{y \in \Z^N} \int_{y+Q} |v^3_n|^{q} \, dx, 
\]
we repeat the same procedure. Thus, may define 
$(d_k)$, $(v^k_n)_{n \in \N}$, $(y^k_n)_{n \in \N}$ and $(\wt u_j)_{0 \leq j \leq k}$ for each $k \geq 1$. 
If $d_k = 0$ holds for some $k \in \N$, then the procedure is terminated and 
the assertions hold for $K:=k$. Therefore, it suffices to consider the case $d_k>0$ for all $k \geq 1$. 
Remark that for any $k \in \N$, 
\[
\int_{Q} | \wt u_{k+2} |^{q} \, dx = d_{k+1} \leq d_k = \int_{Q} |\wt u_{k+1} |^{q} \, dx. 
\]

We next claim that 
\begin{equation}\label{eq:sum-c_k}
	\sum_{j=0}^\infty d_j < \infty. 
\end{equation}
In fact, by the classical Brezis--Lieb lemma (or \cref{lem:BL}), for every $k \in \N$, we have 
\[
\| u_n \|_{q}^{q} = \sum_{j=0}^{k+1} \| \wt u_j \|_{q}^{q} + \| v^{k+1}_j \|_{q}^{q} + o(1). 
\]
Therefore, by $q \in (p_0,p_0^*)$, for each $k \in \N$, 
\[
\sum_{j=0}^{k+1} d_j = \sum_{j=0}^{k} \int_{Q} | \wt u_{j+1} |^{q} \, dx 
\leq \liminf_{n \to \infty} \left[ \| u_n \|_{q}^{q} - \| \wt u_0 \|_{q}^{q}  - \| v^{k+1}_j \|_{q}^{q} \right] \leq \sup_{n \geq 1} \| u_n \|_{q}^{q} < \infty. 
\]
Since $k \in \N$ is arbitrary, \eqref{eq:sum-c_k} holds. In particular, $d_k \to 0$ holds as $k \to \infty$.

To apply \cref{lem:lions} for $(v^k_n)_{n \in \N}$ uniformly with respect to $k \geq 1$, we prove that for every $k \in \N$, 
\begin{equation}\label{bdd-vkn}
	\limsup_{n \to \infty} \| v^k_n \|_X \leq 2 \limsup_{n \to \infty} \| u_n \|_X. 
\end{equation}
For this purpose, let $k \geq 1$. Since 
\[
v^k_n (x) = u_n(x) - \sum_{j=0}^k \wt{u}_j (x-y_n^j), \quad |j_1-j_2| \to \infty \ \ (j_1\neq j_2),
\]
it follows that 
\[
\left\| \sum_{j=0}^k \wt{u}_j (\cdot - y_n^j) \right\|_X \to \sum_{j=0}^k \| \wt{u}_j \|_X.
\]
By a similar argument to $\Psi_0$, we get 
\[
\sum_{j=0}^k \| \wt{u}_j \|_X \leq \liminf_{n \to \infty} \| u_n \|_X. 
\]
Thus, 
\[
\limsup_{n \to \infty} \| v^k_n \|_X 
\leq \limsup_{n \to \infty} \| u_n \|_X + \limsup_{n \to \infty} \left\| \sum_{j=0}^k \wt{u}_j (\cdot - y^j_n) \right\|_X 
\leq 2 \limsup_{n \to \infty} \| u_n \|_X
\]
and \eqref{bdd-vkn} holds.

By \cref{lem:lions}, \eqref{bdd-vkn} and 
\[
\int_{\R^N} \zeta (u_n) \, dx = \sum_{j=0}^k \int_{\R^N} \zeta (\wt u_j) \, dx + \int_{\R^N} \zeta ( v^k_n ) \, dx + o(1), 
\]
for any $\varepsilon >0$ there exist $C_\varepsilon>0$ and $\theta = (q-p_0)/q \in (0,1)$, which does not depend on $k$, such that 
for each $k \geq 1$, 
\[
\limsup_{n \to \infty}\int_{\R^N} \left| \zeta (v^k_n) \right| dx 
\leq \varepsilon + C_\varepsilon d_k^\theta. 
\]
From $d_k \to 0$ as $k \to \infty$, choose $k_\varepsilon \in \N$ so large that 
$C_\varepsilon d_{k}^\theta \leq \varepsilon$ for all $k \geq k_\varepsilon$. Thus, if $k \geq k_\varepsilon$, then 
\[
\limsup_{n \to \infty}\left| \int_{\R^N} \zeta (u_n) \, dx - \sum_{j=0}^k \int_{\R^N} \zeta (\wt u_j) \, dx   \right| \leq 2\varepsilon.
\]
Since we may assume that the limit $\lim_{n \to \infty} \int_{\R^N} \zeta (u_n) \, dx$ exists, 
the above inequality implies 
\[
\lim_{n \to \infty} \int_{\R^N} \zeta(u_n) \, dx = \sum_{j=0}^\infty \int_{\R^N} \zeta (\wt u_j) \, dx. 
\]
For the assertion on $\Psi_0$, for any $k \in \N$ and $R>0$, we may prove 
\[
\liminf_{n \to \infty} \Psi_0(u_n) \geq \sum_{j=0}^k \Psi_{0,B(0,R)}( \wt u_j ).
\]
By letting $R \to \infty$ and then $k \to \infty$, the desired assertion holds. 
\end{proof}

\section{The zero mass case}
\label{Sec:Zeromass}

In this section, the zero mass case is considered and \cref{th:1,th:2} are proved. 
From now on, $N \geq 3$ and \eqref{F1}--\eqref{F6} are always supposed. 
We first investigate properties of $M_0$ and $\cM_0$ defined in \eqref{eq-def-funcM0} and \eqref{eq-def-setM0}, 
and then prove \cref{th:1,th:2}.

\subsection{The projection into the Poho\v{z}aev set}
\label{subsect:M}

For $u \in D(\Psi) \setminus \{0\}$ and $\theta > 0$, we consider the scaling 
\[
u_\theta := \theta u \left( \frac{\cdot}{\theta} \right).
\]
It is immediate to see 
\[
\| \nabla u_\theta \|_q^q = \theta^{N} \| \nabla u \|_q^q \quad \mbox{ for } q \geq 1, \quad \|\nabla u_\theta \|_\infty = \|\nabla u\|_\infty.
\]
From the above observation we see that $u_\theta \to 0$ in $X_0$ as $\theta \to 0$. By a direct calculation we obtain
\begin{equation}\label{eq-diff-theta}
	\begin{aligned}
		I(u_\theta) &= \theta^N \left( \Psi(u) - \Phi(\theta u) \right), \\
		\frac{d}{d\theta} I(u_\theta) &= N \theta^{N-1} \left( \Psi(u) - \Phi(\theta u) - \frac{1}{N} \Phi'(\theta u)(\theta u) \right) = N \theta^{N-1} \left( \Psi(u) - \int_{\R^N} H(\theta u) \, dx \right).
	\end{aligned}
\end{equation}

To study properties of the Poho\v{z}aev set $\cM_0$, we first prove 
\begin{Lem}\label{Lem-Theta}
	For each $\alpha \in (0,\infty)$ and $u \in X_0 \setminus \{0\}$, there exists a unique $\Theta = \Theta(\alpha,u) \in (0,\infty)$ such that 
	\[
	\alpha = \int_{\R^N} H(\Theta u) \, dx. 
	\]
	Moreover, $\Theta \in \cC^1 ( (0,\infty) \times X_0 \setminus \{0\} ; (0,\infty) )$. 
\end{Lem}

\begin{proof}
Define
\[
J(\alpha, u, \theta) := \alpha - \int_{\R^N} H ( \theta u) \, dx :  (0,\infty) \times X_0 \setminus \{0\} \times (0,\infty) \to \R.
\]
In view of $X_0 \subset L^\infty(\R^N)$ and \eqref{F1}--\eqref{F2}, $J$ is of class $\cC^1$. 
For any fixed $u \in X_0 \setminus \{0\}$, it follows from \eqref{F3} that 
\begin{equation}\label{eq-diff-H}
\frac{d}{d \theta} \int_{\R^N} H(\theta u) \, dx = \theta^{-1} \int_{\R^N} H'(\theta u) \theta u \, dx > 0. 
\end{equation}
Furthermore, by \eqref{F2}--\eqref{F4} and $u \in X_0 \setminus \{0\}$, 
\[
\lim_{\theta \to 0} \int_{\R^N} H(\theta u) \, d x = 0, \quad \lim_{\theta \to \infty}\int_{\R^N} H(\theta u) \, dx = \infty.
\]
Therefore, for each $\alpha \in (0,\infty)$ and $u \in X_0 \setminus \{0\}$, 
there exists a unique $\Theta = \Theta(\alpha,u) \in (0,\infty)$ such that $J(\alpha, u, \Theta) = 0$. 
Since 
\[
\frac{\partial J}{\partial \theta} (\alpha, u , \Theta) = - \Theta^{-1} \int_{\R^N} H'( \Theta u ) \Theta u \, dx < 0,
\]
the regularity of $\Theta$ follows from the implicit function theorem. 
\end{proof}

By \cref{Lem-conti-Psi,Lem-Theta}, we may construct a continuous projection from $D(\Psi) \setminus \{0\}$ onto $\cM_0$: 

\begin{Lem}\label{lem:projection}
For every $u \in D(\Psi) \setminus \{0\}$ there is unique $\theta = \theta(u) > 0$ such that $u_\theta \in \cM_0$. Moreover the map
\[
m(u) := u_{\theta(u)} : D(\Psi) \setminus \{0\} \to \cM_0
\]
is continuous, where the topology on $D(\Psi) \setminus \{0\}$ is induced from $X_0$.
\end{Lem}

\begin{proof}
Recall \eqref{eq-def-funcM0} and \eqref{eq-def-setM0}. 
Since 
\[
M_0(u_\theta) = \theta^N \left( \Psi(u) - \int_{\R^N} H(\theta u) \, dx \right),
\]
the desired map $m$ is given by setting $\theta(u) \coloneq \Theta ( \Psi(u) , u )$ 
where $\Theta$ appears in \cref{Lem-Theta}. 
Since $\Psi : D(\Psi) \setminus \{0\} \to \R$ is continuous by \cref{Lem-conti-Psi}, 
$\theta(\cdot)$ is continuous, and hence so is $m$. 
\end{proof}

\begin{Rem}\label{Rem:max}
	Let $u \in D(\Psi) \setminus \{0\}$. 
	From \eqref{eq-diff-theta} and \eqref{eq-diff-H}, 
	$\theta = \theta(u)$ is the only maximum of the function $[0,\infty) \ni \theta \mapsto I(u_\theta)$ with $I(u_{\theta(u)}) > 0$. 
\end{Rem}

\subsection{The ground state level}

In this subsection we are going to show $c_0 = c_{MP} > 0$ where $c_{0,MP}$ is the mountain pass level
\begin{equation}\label{eq:MP}
c_{0,MP} := \inf_{\gamma \in \Ga} \sup_{\tau \in [0,1]} I(\gamma(\tau)), \quad \Ga := \left\{ \gamma \in \cC([0,1]; X_0) \ : \ \gamma(0)=0, \ I(\gamma(1)) < 0 \right\}.
\end{equation}

\begin{Rem}\label{rem:mp}
Fix any $\varphi \in \cC^\infty_0 (\R^N)$ so that $\| \nabla \varphi \|_\infty < 1$. 
Then for each $\theta > 0$, $\varphi_\theta = \theta \varphi(\cdot/ \theta ) \in D(\Psi)$ holds. 
Moreover, as in \cref{subsect:M}, we may prove $\| \varphi_\theta \|_{X_0} \to 0$ as $\theta \to 0$ and 
$I(\varphi_\theta) \to - \infty$ as $\theta \to \infty$. 
Therefore, for a sufficiently large $T>0$, $\gamma_0 (\theta) := \varphi_{T\theta} \in \Gamma$ and 
$c_{0,MP} < \infty$ holds. 
Moreover, it is clear that 
\[
c_{0,MP} = \inf_{\gamma \in \wt \Gamma} \sup_{\tau \in [0,1]} I(\gamma(\tau)), \quad 
\wt \Ga := \left\{ \gamma \in \cC([0,1]; X_0) \ : \ \gamma([0,1]) \subset D(\Psi), \  \gamma(0)=0, \ I(\gamma(1)) < 0 \right\}.
\]
We also remark that for any $\gamma \in \wt \Gamma$, the function $[0,1] \ni t \mapsto I(\gamma(t)) \in \R$ is continuous 
by virtue of \cref{Lem-conti-Psi}. 
\end{Rem}

\begin{Lem}\label{Lem:c0-c0MP}
There holds $c_0 = c_{0,MP} > 0$.
\end{Lem}

\begin{proof}
We first prove $c_0 \geq c_{0,MP}$. 
Fix any $u \in \cM_0$. Then we know that $\theta(u) = 1$. 
Consider a path $\gamma_u (\tau) := u_{T_u \tau}$, where $T_u \gg 1$ and $\gamma_u (0) := \lim_{\tau \to 0^+} u_{T_u \tau} = 0$. 
Note that $I(u_{T_u}) < 0$ for sufficiently large $T_u$, hence $\gamma_u \in \Gamma$. 
Since \cref{Rem:max} leads to $ c_{0,MP} \leq  \max_{\tau \in [0,1]} I(\gamma_u(\tau)) = I(u)$ and $u \in \cM_0$ is arbitrary, 
\[
c_{0,MP} \leq \inf_{u \in \cM_0} I(u) = c_0.
\]
To show $c_{0,MP}>0$ we recall that $X_0 \subset L^\infty(\R^N)$ and $|F(s)| \lesssim |s|^{2^*}$ for $|s| \leq 1$ thanks to (\ref{F2}). 
By \eqref{exp-sq-1t^2}, if $\|u\|_{X_0}$ is sufficiently small, then 
\begin{align*}
I(u) &\geq \sum_{j=1}^\infty b_j \| \nabla u\|_{2j}^{2j} - C \|u\|_{2^*}^{2^*} \\
&\geq  b_1 \| \nabla u \|_2^2 + b_{j_0} \|\nabla u\|_{2j_0}^{2j_0} - C\| \nabla u \|_2^{2^*} \\
&= \left(b_1 - C \| \nabla u\|_2^{2^*-2} \right) \| \nabla u\|_2^2 + b_{j_0} \|\nabla u\|_{2j_0}^{2j_0} \\
&\gtrsim  \|u\|^{2j_0}_{X_0}.
\end{align*}
Hence there is a small radius $\rho > 0$ such that $\inf_{\|u\|_{X_0} =\rho_0} I(u) > 0$ and $I(u) \geq 0$ provided $\| u \|_{X_0} \leq \rho_0$. 
Thus, $c_{0,MP} > 0$.

	We next prove $c_0 \leq c_{0,MP}$. By \cref{rem:mp}, as in \cite{JeTa03}, 
it suffices to show $\gamma ([0,1]) \cap \cM_0 \neq \emptyset$ for any $\gamma \in \wt \Gamma$. 
Let $\gamma \in \wt \Gamma$ and $t_0 := \max \{ t \in [0,1] \ : \ \gamma(t) = 0 \} \in [0,1)$. 
For $M_0$ in \eqref{eq-def-setM0}, as in the above, for sufficiently small $\delta > 0$, we observe that 
\[
M_0(\gamma(t)) > 0 \quad \text{for all $t \in (t_0, t_0+\delta)$}. 
\]
On the other hand, since (\ref{F5}) implies $f(s) s \geq 0$ for any $s \in \R$, it follows that 
$\Phi'(u) u \geq 0$ for every $u \in X_0$ and 
\[
M_0(\gamma(1)) = I(\gamma(1)) - \frac{1}{N}\Phi'(\gamma(1)) \gamma(1) \leq I(\gamma(1)) < 0.
\]
Since the function $[0,1] \ni t \mapsto M_0(\gamma(t)) \in \R$ is continuous by \cref{Lem-conti-Psi}, 
there exists $t_1 \in (t_0,1)$ such that $M_0(\gamma(t_1)) = 0$, which yields 
$\gamma(t_1) \in \cM_0$ and $\gamma([0,1]) \cap \cM_0 \neq \emptyset$. Thus we complete the proof. 
\end{proof}

\subsection{Minimizers on the Poho\v{z}aev set are critical points}

The aim of this subsection is to prove that every minimizer corresponding to $c_0$ is a critical point of $I$, and 
hence \cref{l:RegCritP} implies that it is a classical solution of \eqref{eq:BI}. 

\begin{Lem}\label{lem:minimizers-are-critical}
Let $u_0 \in \cM_0$ satisfy $I(u_0)=c_0$. Then $u_0$ is a critical point of $I$.
\end{Lem}

\begin{proof}
Suppose that $u_0 \in \cM$ satisfies $I(u_0) = c_0 > 0$. 
By \cref{Lem-conti-Psi}, there is $r_0 > 0$ such that
\[
\frac12 \Psi(u_0) \leq \Psi(u) \leq 2 \Psi(u_0)
\]
for $u \in \overline{B_{X_0} (u_0, r_0)} \cap D(\Psi)$. We recall that
\[
\frac{d}{d\theta} I(u_\theta) = N \theta^{N-1} \left( \Psi(u) - \int_{\R^N} H(\theta u) \, dx \right)
\]
and therefore we have the following formula
\[
\frac{d^2}{d\theta^2} I(u_\theta) = N(N-1) \theta^{N-2} \left( \Psi(u) - \int_{\R^N} H(\theta u) \, dx \right) - N \theta^{N-2} \int_{\R^N} H'(\theta u)(\theta u) \, dx.
\]
Hence
\[
L_1 := \sup \left\{ \left| \frac{d^2}{d\theta^2} I(u_\theta) \right| \ : \  \frac12 \leq \theta \leq 2, \ u \in \overline{B_X (u_0, r_0)} \cap D(\Psi) \right\} < \infty.
\]
From $ d I(u_\theta) / d \theta |_{\theta = \theta (u)} = 0$, $\theta(u_0) = 1$ 
and the continuity of $\theta(\cdot)$, it follows that for $u \in \overline{B_{X_0} (u_0, r_0)} \cap D(\Psi)$, 
\begin{equation}\label{eq-2ndorder}
	\frac{1}{2} \leq \theta(u) \leq 2, \quad 
	\left| I(u_{\theta(u)}) - I(u) \right| \leq \left| \int_{\theta(u)}^1 \frac{d}{d\theta} I(u_\theta) \, d\theta \right| \leq \frac{L_1}{2} |1 -\theta(u)|^2. 
\end{equation}

To prove that $u_0$ is a critical point of $I$, let $v \in D(\Psi)$ be any point: otherwise, \eqref{def-cri} holds trivially. 
For $t \in (0,1)$, we introduce
\[
v_t := (1-t) u_0 + t v \in D(\Psi). 
\]
The convexity of $\Psi$ implies that for any $ t \in (0,1]$, 
\begin{align}\label{1}
\Psi(v_t)-\Psi(u_0)-\Phi'(u_0)(v_t-u_0) \leq t \left\{\Psi(v)-\Psi(u_0) -\Phi'(u_0)(v-u_0)\right\}.
\end{align}

Let $(t_n) \subset (0,1)$ satisfy $t_n \to 0$. We first consider the case $\Psi(u_0)\leq \Psi(v_{t_n})$ for all $n$. 
Then it is clear that $v_{t_n} \to u_0$ in $X_0$ and from the convexity of $\Psi$, 
\begin{align}\label{2}
0 \leq \Psi(v_{t_n})-\Psi(u_0) \leq (1-t_n)\Psi(u_0)+t_n \Psi(v)-\Psi(u_0) = t_n \left(\Psi(v)-\Psi(u_0) \right).
\end{align}
To simplify notations we put
\[
\theta_n := \theta(v_{t_n}), \quad w_n := (v_{t_n})_{\theta_n} = \theta_n v_{t_n} \left( \frac{\cdot}{\theta_n} \right) \in \cM_0.
\]
From $\| v_{t_n} - u_0 \|_{X_0} = O(t_n)$ and \eqref{2} 
with $\Theta \in \cC^1((0,\infty) \times X_0 \setminus \{0\} )$ and $\Theta(\Psi(u),u)=\theta(u)$ for $u \in D(\Psi)\setminus \{0\}$, we get 
\begin{equation}\label{eq-diff-th_n-1}
	|\theta_n-1| = |\Theta(\Psi(v_{t_n}), v_{t_n}) - \Theta(\Psi(u_0),u_0)| = O(t_n).
\end{equation}
Using the fact that $\Phi \in \cC^1(X;\R)$ and \eqref{1}, we infer that
\begin{equation}\label{eq-diff-IvnIu0}
	\begin{aligned}
		I(v_{t_n}) - I(u_0) &= \Psi(v_{t_n}) - \Psi(u_0) - (\Phi(v_{t_n})-\Phi(u_0)) \\
		&= \Psi(v_{t_n}) - \Psi(u_0) - \Phi'(u_0)(v_{t_n}-u_0) + o(t_n) \\
		&\leq t_n \left\{\Psi(v)-\Psi(u_0) -\Phi'(u_0)(v-u_0)\right\} + o(t_n).
	\end{aligned}
\end{equation} 
By \eqref{eq-2ndorder} and \eqref{eq-diff-th_n-1} with $v_{t_n} \in \overline{B_{X_0} (u_0, r_0)} \cap D(\Psi) $, 
\begin{align*}
I(w_n)-I(v_{t_n}) \leq \frac{L_1}{2} |\theta_n-1|^2 = O(t_n^2)
\end{align*}
and recalling that $u_0$ minimizes $I$ on $\cM_0$, we see from \eqref{eq-diff-IvnIu0} that 
\[
0 \leq I(w_n) - I(u_0) \leq I(v_{t_n}) - I(u_0) + O(t_n^2) \leq t_n \left\{\Psi(v)-\Psi(u_0) -\Phi'(u_0)(v-u_0)\right\} + o(t_n).
\]
Dividing the above inequality by $t_n$ and sending $t_n \to 0$ imply \eqref{def-cri} under $\Psi(u_0) \leq \Psi(v_{t_n})$ for each $n$.

If there exists a subsequence $(n_k)_{k \in \N}$ such that $\Psi(u_0) \leq \Psi(v_{t_{n_k}})$ for all $k$, 
then the previous argument is still valid. 
Therefore, we may assume $\Psi(u_0) > \Psi(v_{t_n})$ for all $n$. We will use the same notation as in the previous case. If
\[
\limsup_{n\to\infty} \frac{|\theta_n-1|}{t_n} < \infty,
\]
then \eqref{eq-diff-th_n-1} holds in this case and the previous argument can be exploited to obtain 
\[
\Psi(v) - \Psi(u_0) - \Phi'(u_0) (v-u) \geq 0.
\]
Consequently, we may assume that 
\begin{equation}\label{eq-theta_n-1-subl}
	\lim_{n\to\infty} \frac{|\theta_n-1|}{t_n} = \infty.
\end{equation}

In what follows, we shall show that \eqref{eq-theta_n-1-subl} does not occur, 
which implies that the desired inequality is satisfied by any $v \in D(\Psi)$ and this completes the proof. 
For this purpose, we next show $\theta_n \in (0,1)$ for all sufficiently large $n$. Indeed, suppose that there exists $(\theta_{n_k})$ such that 
$n_k \to \infty$ and $\theta_{n_k} \geq 1$. By $\| v_{t_{n_k}} - u_0 \|_{X_0} = c t_{n_k}$ and 
\[
0 = \Psi(v_{t_{n_k}}) - \int_{\R^N} H( \theta_{n_k} v_{t_{n_k}} ) \, dx, 
\]
it follows from the facts $u_0 \in \cM_0$ and $\Psi(u_0) > \Psi( v_{t_n} )$ in view of the assumption that for some $C_0>0$, which is independent of $k$, 
\[
\begin{aligned}
	\int_{\R^N} H(u_0) \, dx = \Psi(u_0) > \Psi(v_{t_{n_k}}) = \int_{\R^N} H(\theta_{n_k} v_{t_{n_k}}) \, dx 
	\geq \int_{\R^N} H(\theta_{n_k} u_0) \, dx - C_0 t_{n_k}. 
\end{aligned}
\]
Thus, we may find $\tau_k \in (0,1)$ such that 
\begin{equation}\label{eq:cont}
C_0 t_{n_k}  \geq  \int_{\R^N} H(\theta_{n_k} u_0) - H(u_0) \, dx 
= \int_{\R^N} H' \left( \left( 1 + \tau_k (\theta_{n_k} -1) \right) u_0 \right) u_0 \, dx \left( \theta_{n_k} - 1 \right).
\end{equation}
However, since $\theta_n = \theta(v_{t_n}) \to 1$ and \eqref{F3} lead to 
\[
\lim_{k \to \infty}\int_{\R^N} H' \left( \left( 1 + \tau_k (\theta_{n_k} -1) \right) u_0 \right) u_0 \, dx 
= \int_{\R^N} H'(u_0) u_0 \, dx > 0,
\]
\eqref{eq:cont} contradicts \eqref{eq-theta_n-1-subl}. Therefore, $\theta_n \in (0,1)$ holds for all sufficiently large $n$.

Since $\theta_n < 1$ and $I(u) = \frac{1}{N} \int_{\R^N} f(u)u \, dx$ for $u \in \cM_0$, \eqref{F5} and the fact $w_n \in \cM_0$ give 
\begin{align}
I(u_0) \leq I(w_n) &= \frac{\theta_n^N}{N} \int_{\R^N} f(\theta_n v_{t_n}) \theta_n v_{t_n} \, dx \nonumber \\
&\leq \frac{\theta_n^N}{N} \int_{\R^N} f( v_{t_n}) v_{t_n} \, dx \label{3} \\
&= \frac{1}{N} \int_{\R^N} f( v_{t_n}) v_{t_n} \, dx - (1-\theta_n^N) \frac{1}{N} \int_{\R^N} f( v_{t_n}) v_{t_n} \, dx. \nonumber
\end{align}
Observe that
\begin{equation*}
	\lim_{n\to\infty} \frac{1 - \theta_n^N }{t_n} \geq \lim_{n\to\infty} \frac{1-\theta_n}{t_n} = \infty
\end{equation*}
and that
\[
\int_{\R^N} f(v_{t_n})v_{t_n} \, dx \to \int_{\R^N} f(u_0)u_0 \, dx = N I(u_0) = Nc_0 > 0.
\]
Moreover by $\| v_{t_n} - u_0 \|_{X_0} = t_n \| v - u_0 \|_{X_0} = O(t_n)$ and 
the fact that $X_0 \ni u \mapsto \int_{\R^N} f(u)u \, dx \in \R$ is of class $\cC^1$, for some $C_1>0$, 
\[
\left| N I(u_0) - \int_{\R^N} f(v_{t_n}) v_{t_n} \, dx \right| = \left| \int_{\R^N} f(u_0)u_0 \, dx - \int_{\R^N} f(v_{t_n}) v_{t_n} \, dx \right| \leq C_1 t_n.
\]
However this contradicts \eqref{3}. Indeed, from \eqref{3} we obtain 
\[
I(u_0) - \frac{1}{N} \int_{\R^N} f( v_{t_n}) v_{t_n} \, dx \leq - (1-\theta_n^N) \frac{1}{N} \int_{\R^N} f( v_{t_n}) v_{t_n} \, dx
\]
or equivalently
\[
(1-\theta_n^N) \frac{1}{N} \int_{\R^N} f( v_{t_n}) v_{t_n} \, dx \leq \frac{1}{N} \int_{\R^N} f( v_{t_n}) v_{t_n} \, dx - I(u_0) \leq \frac{C_1}{N} t_n. 
\]
This yields the following contradiction: 
\[
0 < \int_{\R^N} f(u_0) u_0 \, dx = \lim_{n \to \infty} \int_{\R^N} f( v_{t_n} ) v_{t_n} \, dx \leq C_1 \limsup_{n\to\infty} \frac{t_n}{1-\theta_n^N} = 0.
\]
Thus \eqref{eq-theta_n-1-subl} does not occur and $u_0$ is a critical point of $I$. 
\end{proof}

\subsection{Existence of a minimizer and other properties of minimizers}

To prove \cref{th:1}, from \cref{l:RegCritP,Lem:c0-c0MP,lem:minimizers-are-critical}, it suffices to prove that 
there exists a minimizer, any minimizer has a constant sign and there exists a radial minimizer corresponding to $c_0$. 
We will show the existence of a minimizer by a direct minimization technique on $\cM_0$ and the profile decomposition.

\begin{proof}[Proof of \cref{th:1}]
Let $(u_n)_{n \in \N} \subset \cM_0$ be a minimizing sequence, namely $I(u_n) \to c_0$. Note that for $u \in \cM_0$ we have
\[
I(u) = \frac{1}{N} \int_{\R^N} f(u)u \, dx.
\]
Thus, $(\int_{\R^N} f(u_n) u_n \, dx  )_{n \in \N}$ is bounded. From \eqref{F6}, $( \int_{\R^N} F(u_n) \, dx )_{n \in \N}$ is bounded and 
by $I(u_n) \to c_0$, $(\Psi(u_n))_{n \in \N}$ is also bounded. 
Hence $(u_n)_{n \in \N}$ is bounded in $X_0$ via \eqref{exp-sq-1t^2}. From \cref{Prop:decomp} with $p_0=2^*$, 
there exist $K \in \{0,1,2,\dots\} \cup \{\infty\}$, $(y_n^j)_n$ and $\wt u_j$ such that 
\begin{equation}\label{eq-dec}
	\begin{aligned}
		&u_n(\cdot + y_n^j) \weakto \widetilde{u}_j, \quad u_n(x+y_n^j) \to \wt{u}_j (x) \quad \text{a.e. $\R^N$}, \\
		&\lim_{n\to\infty} \Psi(u_n) \geq \sum_{j=0}^K \Psi(\widetilde{u}_j) , \\
		&\lim_{n \to \infty} \int_{\R^N} F(u_n) \, dx = \sum_{j=0}^K \int_{\R^N} F(\widetilde{u}_j) \, dx, \\
		&\lim_{n \to \infty} \int_{\R^N} f(u_n)u_n \, dx = \sum_{j=0}^K \int_{\R^N} f(\widetilde{u}_j) \widetilde{u}_j \, dx.
	\end{aligned}
\end{equation}

If $K=0$ and $\wt u_0 = 0$, then 
\[
I(u_n) = \frac{1}{N} \int_{\R^N} f(u_n)u_n \, dx \to 0
\]
and $c_0 = 0$, which is a contradiction. Hence either $K \geq 1$ or $K=0$ and $\wt u_0 \not \equiv 0$.

Observe that $\cM_0 \subset D(\Psi) \setminus \{0\}$ due to $M_0(u) = \infty$ for $u \not \in D(\Psi)$, 
hence $u_n \in D(\Psi) \setminus \{0\}$. 
Since $D(\Psi)$ is convex and closed, it is also weakly closed thanks to Mazur's theorem, which gives $\tu_j \in D(\Psi)$.

To treat the cases $K \geq 1$ and $K = 0$ with $\wt u_0 \not \equiv 0$ simultaneously, set $J := \left\{ j \ : \ \wt u_j \not \equiv 0 \right\}$. 
Remark that $J \neq \emptyset$ by the above observation. 
For $j \in J$, let $\theta_j > 0$ be the unique number so that $(\tu_j)_{\theta_j} \in \cM_0$. 
Suppose that for all $j \in J$ we have $\theta_j > 1$. By \eqref{eq-diff-H} and \eqref{eq-dec}, 
\begin{equation}\label{Ineq:Psi}
	\begin{aligned}
		0 < \sum_{j \in J} \Psi(\tu_j) \leq \lim_{n\to\infty} \Psi(u_n) = \lim_{n\to\infty} \int_{\R^N} H(u_n) \, dx &= \sum_{j \in J} \int_{\R^N} H(\widetilde{u}_j) \, dx \\
		&< \sum_{j \in J} \int_{\R^N} H(\theta_j \widetilde{u}_j) \, dx = \sum_{j \in J} \Psi(\tu_j),
	\end{aligned}
\end{equation}
which is a contradiction. Hence there exists $j_0 \in J$ such that $\theta_{j_0} \leq 1$. Then, Fatou's lemma and \eqref{F5} yield
\begin{equation}\label{Ineq:energy}
	\begin{aligned}
	c_0 = \lim_{n\to\infty} I(u_n) &= \lim_{n\to\infty} \frac{1}{N} \int_{\R^N} f(u_n)u_n \, dx 
	= \lim_{n\to\infty} \frac{1}{N} \int_{\R^N} f(u_n(\cdot + y_n^{j_0} )) u_n(\cdot + y_n^{j_0} ) \, dx \\
	&\geq \frac{1}{N} \int_{\R^N} f(\tu_{j_0})\tu_{j_0} \, dx \geq \frac{1}{N} \theta_{j_0}^N \int_{\R^N} f(\theta_{j_0} \tu_{j_0}) \theta_{j_0} \tu_{j_0} \, dx 
	= I(m(\tu_{j_0})) \geq c_0.
	\end{aligned}
\end{equation}
Thus $\theta_{j_0} = 1$, $\tu_{j_0} \in \cM_0$, $\lim_{n \to \infty} \int_{\R^N} f(u_n) u_n \, dx = \int_{\R^N} f( \wt{u}_{j_0} ) \wt{u}_{j_0} \, dx$ 
and $I(\tu_{j_0}) = c_0$. Moreover, since $f(s)s > 0$ for $s \neq 0$ in view of \eqref{F3} and \eqref{F5}, 
by \eqref{eq-dec}, $\tu_j = 0$ for each $j \neq j_0$, $J=\{j_0\}$ and $\lim_{n \to \infty} \int_{\R^N} F(u_n) \, dx = \int_{\R^N} F( \wt{u}_{j_0} ) \, dx$. 
Furthermore, by $I(u_n) \to c_0 = I(\tu_{j_0})$, 
we have $\Psi(u_n) \to \Psi( \tu_{j_0} )$, and \cref{l:strconv} yields $\| u_n (\cdot + y_{n}^{j_0}) - \tu_{j_0} \|_{X_0} \to 0$. 
Thanks to \cref{lem:minimizers-are-critical} we know that $\tu_{j_0}$ is a critical point of $I$ 
and \cref{l:RegCritP} implies that $\tu_{j_0}$ is a classical solution of \eqref{eq:BI}.

We next prove that all minimizers have a constant sign. 
To this end, let $u \in \cM_0$ be any minimizer corresponding to $c_0$. 
We argue by contradiction and suppose $u_\pm \not \equiv 0$ where $u_+ := \max \{u,0\}$ and $u_- := \min \{u,0\}$. 
From $\| \nabla u \|_\infty < 1$ due to \cref{l:RegCritP,lem:minimizers-are-critical}, it follows that 
$u_+,u_- \in D(\Psi) \setminus \{0\}$. According to \cref{lem:projection} there exist $\theta_+$ and $\theta_-$ such that 
$(u_+)_{\theta_+}, (u_-)_{\theta_-} \in \cM_0$. By 
\[
\Psi(u_+) + \Psi(u_-) = \Psi(u) = \int_{\R^N} H(u) \, dx =  \int_{\R^N} H(u_+) \,dx + \int_{\R^N} H(u_-) \, dx,
\]
arguing as in \eqref{Ineq:Psi}, we get either $\theta_+\leq 1$ or $\theta_- \leq 1$. 
Let $\theta_+ \leq 1$. By noting $f(s) s \geq 0$ and replacing $u_n$ and $\wt{u}_{j_0}$ to $u$ and $u_+$ in \eqref{Ineq:energy}, 
we deduce that $\theta_+ = 1 $, $u_+ \in \cM_0$, $I(u_+) = c_0$ and $u_- \equiv 0$, which is a contradiction. 
Similarly, a contradiction occurs when $\theta_- \leq 1$ holds. 
Therefore, either $u_+ \equiv 0$ or $u_- \equiv 0$. Since $u$ is a classical solution of \eqref{eq:BI} with either $u \geq 0$ or $u \leq 0$, 
the strong maximum principle gives either $u>0$ or $u<0$ and thus $u$ has a constant sign.

Finally, we prove the existence of a radial minmizer corresponding to $c_0$. 
Let $u \in \cM_0$ be arbitrary minimizer of $c_0$. 
From the above, either $u>0$ or $u<0$ holds. 
Set $v:= u^*$ when $u>0$ where $u^*$ is the symmetric decreasing rearrangement of $u$, 
and $v:= - u^*$ when $u<0$. 
From \cite[Lemma 1]{Ta76} and \eqref{exp-sq-1t^2} (or \cite[Lemma 4.3]{BoDeCoDe12}) it follows that 
\[
\Psi(u) = \sum_{j \in \N} b_j \| \nabla u \|_{2j}^{2j} \geq \sum_{j \in \N} b_j \| \nabla u_0^* \|_{2j}^{2j} = \Psi(v).
\]
On the other hand, since $u$ has a constant sign and $f \in \cC^1$, we also have 
\[
\int_{\R^N} F(u) \, dx = \int_{\R^N} F(v) \, dx, \quad 
\int_{\R^N} f(u) u \, dx = \int_{\R^N} f(v) v \, dx,
\]
and hence 
\[
0 = M_0(u) = \Psi(u) - \int_{\R^N} H(u) \, dx \geq \Psi(v) - \int_{\R^N} H(v) \, dx = M_0(v). 
\]
From \cref{Lem-Theta} and \eqref{eq-diff-H}, there exists $\theta \in (0,1]$ such that $v_{\theta} \in \cM_0$ and 
\eqref{F5} yields 
\[
c_0 \leq I( v_{\theta} ) = \frac{\theta^N}{N} \int_{\R^N} f( \theta v ) \theta v \, dx
\leq \frac{1}{N} \int_{\R^N} f (v) v \, dx = \frac{1}{N} \int_{\R^N} f(u) u \, dx = I(u) = c_0.
\]
Hence $\theta = 1$ and $v \in \cM_0$ is also a minimizer and this completes the proof. 
\end{proof}


\subsection{Sobolev-type inequality}

\begin{proof}[Proof of \cref{th:2}]
Consider the case $f(u) = |u|^{p-2} u$ with $p > 2^*$. 
Fix any $u \in D(\Psi) \setminus \{0\}$. 
Then, from above considerations, it is clear that $I(m(u)) \geq c_0$. 
Recalling that $m(u) = u_{\theta(u)} \in \cM_0$, where $\theta = \theta(u) = \Theta(\Psi(u), u)$, we get
\begin{align}\label{ineq:1}
c_0 \leq I(u_\theta) =  \frac{\theta^{N}}{N} \int_{\R^N} f \left( \theta u \right) \theta u \, dx
= \frac{\theta^{N+p}}{N} \| u \|_p^p. 
\end{align}
Notice that the equality happens if and only if $m(u) = u_\theta$ is the minimizer on $\cM_0$. 
Since $F(u) = \frac{1}{p} |u|^p$, $H(u) = \left(\frac{1}{p} + \frac{1}{N} \right) |u|^p = \frac{N+p}{Np} |u|^p$ and $M_0(m(u)) = 0$, it follows that 
\[
\Psi(u) = \int_{\R^N} \frac{N+p}{Np} | \theta u|^p \, dx =\frac{N+p}{Np} \theta^p \int_{\R^N}  |  u|^p \, dx
\]
and that 
\[
\theta = \left( \frac{Np}{N+p} \Psi(u) \right)^{1/p} \frac{1}{\|u\|_p}.
\]
Thus, from \eqref{ineq:1} we obtain
\begin{align*}
c_0 \leq \frac{\theta^{N+p}}{N} \| u\|_p^p  = \frac{1}{N} \left( \frac{Np}{N+p} \Psi(u) \right)^{(N+p)/p} \frac{1}{\|u\|^{N}_p}
\end{align*}
and
\[
\left( \frac{N+p}{Np} \right)^{N+p} N^p c_0^p \| u \|_p^{Np} \leq \Psi(u)^{N+p}.
\]
Furthermore, the equality holds if and only if $m(u)=u_\theta$ is the minimizer of $\inf_{\cM_0} I$, 
and the proof is completed. 
\end{proof}

\section{Positive mass case}
\label{sec:positive}

In this section, the positive mass case is considered. 
Hereafter, $N \geq 1$ and \eqref{G1}--\eqref{G3} are assumed. 
For arguments which are similar to the zero mass case, 
we outline points where we need modifications from the zero mass case.

\subsection{The projection into the Poho\v{z}aev set}

We will show that there exists a continuous projection from $D(\Psi)\setminus \{0\}$ onto $\cM_1$. 
To this end, we first prove 
\begin{Lem}\label{Lem-Theta-pos}
	For every $\alpha > 0$ and $u \in X_1 \setminus \{0\}$ there is a unique $\Theta = \Theta(\alpha, u) > 0$ such that
	\[
	\alpha = \int_{\R^N} - \left( \frac12 + \frac{1}{N} \right) \Theta^2 u^2 + G(\Theta u) + \frac{1}{N} g(\Theta u)\Theta u \, dx.
	\]
	Furthermore, the map $(0,\infty) \times X_1 \setminus \{0\} \ni (\alpha , u) \mapsto \Theta(\alpha,u) \in (0,\infty)$ is of class $\cC^1$. 
\end{Lem}

\begin{proof}
Set 
\[
\begin{aligned}
	\cH(u,\theta) &:= \int_{\R^N} - \left( \frac12 + \frac{1}{N} \right) \theta^2 u^2 + G(\theta u) + \frac{1}{N} g(\theta u)\theta u \, d x, 
	\\
	J(\alpha,u,\theta) &:= \alpha - \cH(u,\theta) \in \cC^1 \left( (0,\infty) \times X_1 \setminus\{0\} \times (0,\infty) ; \R \right).
\end{aligned}
\]
From $X_1 \subset L^\infty(\R^N)$, \eqref{G1} and \eqref{G3}, it follows that for each fixed $u \in X_1 \setminus \{0\}$, 
\begin{equation}\label{eq-lim-theta}
\lim_{\theta \to 0} \left\| \frac{G(\theta u)}{\theta^2 u^2} + \frac{1}{N} \frac{g(\theta u)}{\theta u} \right\|_{\infty} = 0, \quad 
\lim_{\theta \to \infty} \cH(u,\theta) = \infty. 
\end{equation}
Furthermore, by \eqref{G2}, for each $s \neq 0$, 
\begin{equation}\label{eq-2Gg}
G(s) = \int_0^s g(\tau) \, d \tau = \int_0^s \frac{g(\tau)}{\tau} \cdot \tau \, d \tau < \frac{g(s)}{s} \int_0^s \tau \, d \tau = \frac{1}{2}g(s) s. 
\end{equation}
For any $u \in X_1 \setminus \{0\}$ and  $\theta \in (0,\infty)$ with $\cH(u,\theta) \geq 0$, notice that \eqref{eq-2Gg} and \eqref{G2} yield 
\begin{equation}\label{eq-diff-cH}
	\begin{aligned}
		\frac{\partial}{\partial\theta} \cH(u,\theta) 
		&= \int_{\R^N} - 2 \left( \frac12 + \frac{1}{N} \right) \theta u^2 + g(\theta u)u + \frac{1}{N} \left( g'(\theta u)\theta u^2 + g(\theta u)  u \right) \, dx \\
		&= \theta^{-1} \left( 2\cH(u,\theta) +  \int_{\R^N} -2 G(\theta u) + \left(1 - \frac{1}{N}\right) g(\theta u) \theta u + \frac{1}{N} g'(\theta u) \theta^2 u^2 \, dx \right) \\
		&> \frac{\theta^{-1}}{N} \int_{\R^N} g'(\theta u) \theta^2 u^2 - g(\theta u) \theta u \, dx \geq 0. 
	\end{aligned}
	\end{equation}
From \eqref{eq-lim-theta} and this fact, there exists a unique $\Theta_1 = \Theta_1 (u) > 0$ such that 
$\cH(u,\Theta_1) = 0$ and $(\Theta_1,\infty) \ni \theta \mapsto \cH(u,\theta)$ is strictly increasing. 
Therefore, for any given $\alpha > 0$ and $u \in X_1 \setminus \{0\}$,  we may find a unique $\Theta=\Theta(\alpha,u) \in (\Theta_1,\infty)$ such that 
$J(\alpha, u, \Theta) = \alpha - \cH(u,\Theta) = 0$. 
Moreover, by \eqref{eq-diff-cH}, we have 
\[
\frac{\partial J}{\partial \theta} (\alpha, u , \Theta) < 0.
\]
Hence, the regularity of $\Theta$ follows from the implicit function theorem. 
\end{proof}

As in the zero mass case, consider $u_\theta := \theta u(\cdot/\theta)$. Observe that
\begin{align*}
\frac{d}{d\theta} I(u_\theta) = N \theta^{N-1} \left( \Psi(u) - \int_{\R^N} - \left( \frac12 + \frac{1}{N} \right) \theta^2 u^2 + G(\theta u) + \frac{1}{N} g(\theta u)\theta u \, dx \right).
\end{align*}
By putting $\theta(u) := \Theta (\Psi(u), u)$, the following hold from \cref{Lem-conti-Psi,Lem-Theta-pos}: 

\begin{Lem}
For every $u \in D(\Psi) \setminus \{0\}$ there is a unique $\theta = \theta(u) > 0$ such that $u_\theta \in \cM_1$. Moreover the map
\[
m : D(\Psi) \setminus \{0\} \rightarrow \cM_1, \quad m(u) := u_{\theta(u)}
\]
is continuous where the topology on $D(\Psi) \setminus \{0\}$ is induced from $X_1$.
\end{Lem}

\subsection{Minimizers on the Poho\v{z}aev set are critical points}

For $u \in \cM_1$, from $M_1(u) = 0$ it follows that 
\begin{equation}\label{eq:I-on-M1}
	I(u) = I(u) - \frac{N}{N+2} M_1(u) 
	= \frac{2}{N+2} 
	\left\{ \Psi(u) + \int_{\R^N} \frac{1}{2} g(u) u - G(u) \, dx \right\}.
\end{equation}
From \eqref{eq-2Gg}, the integrand of the right-hand side is nonnegative. Moreover, for each $s \in \R$ and $\theta > 0$, 
\eqref{G2} implies 
\[
\theta \frac{d}{d\theta} \left( \frac{1}{2} g(\theta u) \theta u - G(\theta u) \right) 
= \frac{1}{2} g'(\theta u) (\theta u)^2 - \frac{1}{2} g(\theta u) \theta u \geq 0.
\]
In particular, if $u \in D(\Psi) \setminus \{0\}$ and $\theta > 0$, then 
\begin{equation}\label{eq-mono-int}
	\frac{d}{d\theta} \int_{\R^N} \frac{1}{2} g(\theta u) \theta u - G(\theta u) \, dx \geq 0. 
\end{equation}

To prove that minimizers on $\cM_1$ are critical points of $I$, we need the following: 

\begin{Lem}\label{lem:nehari-type-inequality}
Let $f(s) := - s + g(s)$ and $u \in X_1 \setminus \{0\}$ satisfy $\int_{\R^N} f(u)u \, dx > 0$. Then for each $t >1$, 
\[
\int_{\R^N} f(tu)tu \, dx > \int_{\R^N} f(u)u \, dx.
\]
\end{Lem}

\begin{proof}
Denote $\beta := \int_{\R^N} f(u)u \, dx > 0$. For $t > 1$ we observe that
\begin{align*}
\frac{d}{dt} \int_{\R^N} f(tu)tu \, dx &= \int_{\R^N} -2tu^2 + g'(tu)tu^2 + g(tu)u \, dx \\
&= t \left( 2\beta - 2 \int_{\R^N} g(u)u \, dx + \int_{\R^N} g'(tu)u^2 + \frac{g(tu)}{tu} u^2 \, dx \right) \\
&= t \left\{ 2\beta + \int_{\R^N} u^2 \left( - 2 \frac{g(u)}{u} + g'(tu) + \frac{g(tu)}{tu} \right) \, dx \right\}. 
\end{align*}
Here \eqref{G2} and $g'(s) s^2 - g(s) s \geq 0$ imply 
\[
\frac{d}{dt} \int_{\R^N} f(tu)tu \, dx \geq t \left\{ 2\beta +  2 \int_{\R^N} u^2 \left( \frac{g(tu)}{tu} - \frac{g(u)}{u^2} \right)  dx \right\} \geq 2 \beta t > 0
\]
and the proof is completed.
\end{proof}

We next define the mountain pass $c_{1,MP}$ by 
\[
c_{1,MP} := \inf_{\gamma \in \Gamma} \max_{t \in [0,1]} I(\gamma(t)), \quad 
\Gamma := \left\{ \gamma \in \cC( [0,1] ; X_1 ) \ : \ \gamma([0,1]) \subset D(\Psi), \ \gamma(0) = 0, \ I(\gamma(1)) < 0 \right\}. 
\]

\begin{Lem}\label{Lem:c1-c1MP}
	For each $u \in \cM_1$ there exists a path $\gamma_{u} \in \cC([0,1] ; D (\Psi) )$ such that 
	$\max_{0 \leq t \leq 1} I(\gamma_{u}(t)) = I(u)$. 
	Moreover, $c_1 = c_{1,MP} > 0$ holds. 
\end{Lem}

\begin{proof}
Let $u \in \cM_1$ and consider $\gamma_u( t ) := (u)_{T t} = (Tt) u(\cdot/(Tt))$. 
It is immediate to see $\gamma \in \cC( [0,1] ; D(\Psi)  )$. Recalling the notation $\cH$ in the proof of \cref{Lem-Theta-pos}, we see that 
\[
\frac{d}{dt} I(\gamma_u(t)) = NT (Tt)^{N-1} \left( \Psi(u) - \cH(u,tT) \right). 
\]
According to \cref{Lem-Theta-pos} and $u \in \cM_1$, for sufficiently large $T\gg1$, 
\[
\max_{ 0 \leq t \leq 1 } I(\gamma_u(t)) = I(u), \quad I(\gamma_u(1)) < 0. 
\]
Hence, the first assertion holds and this yields $c_{1,MP} \leq c_1$.

Next, we prove the inequality $c_{1,MP} > 0$. 
If $|s| \ll 1$, then \eqref{G1} yields 
\[
F(s) = - \frac{s^2}{2} + G(s) \leq - \frac{s^2}{2} + \frac{s^2}{4} = - \frac{1}{4} s^2.
\]
Thus, when $\| u \|_{X_1} \ll 1$, by $X_1 \subset L^\infty(\R^N)$, we have 
\[
\Phi(u)\leq - \frac{1}{4} \| u \|_2^2, \quad 
I(u) \geq b_1 \| \nabla u \|_2^2 + b_{2j_0} \| \nabla u \|_{2j_0}^{2j_0} + \frac{1}{4} \| u \|_2^2. 
\]
Hence, $c_{1,MP} > 0$ holds.

Finally, the inequality $c_1 \leq c_{1,MP}$ will follow from 
$\gamma([0,1]) \cap \cM_1 \neq \emptyset$ for any $\gamma \in \Gamma$. 
This assertion can be proved by the following argument. 
A similar argument to the above yields $M_1(u) > 0$ for each $0< \| u \|_{X_1} \ll 1$. On the other hand, by  
\[
M_1(\gamma(1)) = I(\gamma(1)) - \frac{1}{N} \Phi'(\gamma(1)) \gamma(1)
\]
and if $u \in D(\Psi) \setminus \{0\}$ satisfies $I(u) < 0$, then \eqref{eq-2Gg} gives 
\[
\Phi'(u) u = \int_{\R^n} f(u) u \, dx = \int_{\R^N } -u^2 + g(u)u \, dx \geq 2 \int_{\R^N} - \frac{u^2}{2} + G(u) \, dx = 2\left( - I(u) + \Psi(u) \right) > 0.
\]
Hence, $M_1(\gamma(1)) < 0$ and there exists $\tau \in (0,1)$ such that $M_1(\gamma(\tau_1)) = 0$. 
This completes the proof. 
\end{proof}

\begin{Lem}\label{lem:positive-mass-minimizers-crit}
Let $u_0 \in \cM_1$ satisfy $I(u_0) = c_1$. Then $u_0$ is a critical point of $I$.
\end{Lem}

\begin{proof}
Set 
\[
H(s) := F(s) + \frac{1}{N} f(s) s = - \left( \frac{1}{2} + \frac{1}{N} \right) s^2 + G(s) + \frac{1}{N} g(s) s. 
\]
Then the proof of \cref{lem:minimizers-are-critical} works until \eqref{eq-theta_n-1-subl}. 
After \eqref{eq-theta_n-1-subl}, we need to prove $\theta_n < 1$ under $\Psi(v_{n}) < \Psi(u_0)$. 
Recall the notation $\cH$ in the proof of \cref{Lem-Theta-pos}. 
Since $u_0 \in \cM_1$ gives 
\[
\cH(u_0,1) = \int_{\R^N} H(u_0) \, dx = \Psi(u_0) > 0, 
\]
\eqref{eq-diff-cH} and \eqref{G2} lead to 
\[
\int_{\R^N} H'(u_0) u_0 \, dx = \frac{\partial \cH}{\partial \theta} (u_0,1) > \frac{1}{N} \int_{\R^N} g'(u_0) u_0^2 - g(u_0)u_0 \,dx \geq 0. 
\]
Hence, the argument for proving $\theta_n < 1$ also works in the positive mass case.

Finally, if $\theta_n < 1$ and $w_n = \theta_n v_{t_n} (x / \theta_n)$, thanks to $0< c_1 = I(u_0)$ and  \cref{lem:nehari-type-inequality},
\[
\begin{aligned}
	I(u_0)\leq I(w_n) = \frac{\theta_n^N}{N} \int_{\R^N} f(\theta_n v_{t_n}) \theta_n v_{t_n} \, dx 
	&\leq \frac{\theta_n^N}{N} \int_{\R^N} f(v_{t_n}) v_{t_n} \, dx
	\\
	&= \frac{1}{N}\int_{\R^N} f(v_{t_n}) v_{t_n} \, dx 
	- \frac{( 1 - \theta_n^N)}{N} \int_{\R^N} f(v_{t_n}) v_{t_n} \, dx. 
\end{aligned}
\]
By $0 < c_1 = I(u_0) = \int_{\R^N} f(u_0) u_0 \, dx / N$, 
we can repeat the same reasoning as for the zero mass case in \cref{lem:minimizers-are-critical} to show that every minimizer of $I$ on $\cM_1$ is a critical point.
\end{proof}

\subsection{Proof of \cref{th:3}}

\begin{proof}[Proof of \cref{th:3}]
Remark that as in the proof of \cref{Lem:c1-c1MP}, for a minimzing sequence $(u_n)_{n \in \N} \subset \cM_1$, by $M_1(u_n) = 0$, we obtain 
\[
NI(u_n) = \Phi'(u_n)u_n = \int_{\R^N} f(u_n)u_n \, dx \geq 2 \int_{\R^N} F(u_n) \, dx =  2 \left( -I(u_n) + \Psi(u_n) \right).  
\]
Hence, $(\Psi(u_n))_{n \in \N}$, $(\Phi'(u_n)u_n)_{n \in \N}$ and $(\Phi(u_n))_{n \in \N}$ are bounded according to $I(u_n) = c_1 + o(1)$. 
When $N \geq 3$, by $X_1 \subset X_0 \subset L^{2^*} (\R^N) \cap L^\infty(\R^N)$, $(u_n)_{n \in \N}$ is bounded in $L^{2^*} (\R^N) \cap L^\infty(\R^N)$. 
From \eqref{G1}, we may find $C>0$, which depends on $\sup_n \| u_n \|_\infty$ such that 
\[
G(u_n(x)) \leq \frac{1}{4} u_n^2(x) + C |u_n(x)|^{2^*} \quad \text{for all $n$ and $x \in \R^N$}.
\]
Hence, 
\[
c_1 + o(1) = I(u_n) \geq \frac{1}{2} \| u_n \|_2^2 - \int_{\R^N} G(u_n) \, dx 
\geq \frac{1}{4} \| u_n \|_2^2 - C \| u_n \|_{2^*}^{2^*} 
\geq \frac{1}{4} \| u_n \|_2^2 - C \| \nabla u_n \|_{2}^{2^*}.
\]
By $\| \nabla u_n \|_2^2/ 2 + b_{2j_0} \| \nabla u_n \|_{2j_0}^{2j_0} \leq \Psi(u_n)$, $(u_n)_{n \in \N}$ is bounded in $L^2(\R^N)$, hence in $X_1$.

On the other hand, when $N =1,2$, we claim that \eqref{G2} and \eqref{G3} imply that 
\begin{equation}\label{eq:g2Ginfi}
\lim_{|s| \to \infty} \left( g(s) s - 2G(s) \right) = \infty.
\end{equation}
In fact, by \eqref{G3},  we may choose $(s_n)_{n \in \N}$ so that 
\[
s_n < s_{n+1}, \quad s_n \to \infty, \quad \frac{g(s_n)}{s_n} + 1 \leq \frac{g(s)}{s} \quad \text{for $s \geq s_{n+1}$}.
\]
Hence, if $s \geq s_{n+1}$, then 
\[
\begin{aligned}
	2G(s) &= 2 \int_{0}^{s_n} \frac{g(\tau)}{\tau} \tau \, d \tau + 2 \int_{s_n}^{s} \frac{g(\tau)}{\tau} \tau \, d \tau
	\\
	&\leq 2 \frac{g(s_n)}{s_n} \int_0^{s_n} \tau \, d \tau + 2 \frac{g(s)}{s} \int_{s_n}^{s} \tau \, d \tau 
	\\
	&= \left( \frac{g(s_n)}{s_n} - \frac{g(s)}{s} \right) s_n^2 +  g(s) s. 
\end{aligned}
\]
The choice of $(s_n)_{n \in \N}$ and $s \geq s_{n+1}$ lead to 
\[
g(s) s - 2G(s) \geq \left( \frac{g(s)}{s} - \frac{g(s_n)}{s_n} \right) s_n^2 \geq s_n^2,
\]
which implies $g(s) s - 2G(s) \to \infty$ as $s \to \infty$. 
Similarly, we may verify $g(s) s - 2G(s) \to \infty$ as $s \to - \infty$ and 
\eqref{eq:g2Ginfi} holds.

Recalling that 
\[
\int_{\R^N} g(u_n)u_n - 2 G(u_n) \, dx = \Phi'(u_n) u_n - 2\Phi(u_n) 
\]
is bounded and that each $u_n$ is $1$-Lipschitz in view of $u_n \in D(\Psi)$, we observe from \eqref{eq-2Gg} and \eqref{eq:g2Ginfi} that 
$(u_n)_{n \in \N}$ is bounded in $L^\infty(\R^N)$. Furthermore, by \eqref{eq-2Gg} and \eqref{eq:g2Ginfi}, for each $t>0$ there exists $\delta_t>0$ such that 
$g(s)s - 2G(s) \geq \delta_t$ if $|s| \geq t$. Therefore, $(|\{ |u_n| \geq t \}|)_{n \in \N}$ is also bounded where $|A|$ denotes 
the $N$-dimensional Lebesgue measure of $A \subset \R^N$. 
By \eqref{G1}, choose $t_0 > 0$ so that $|G(t)| \leq t^2/4$ for each $|t| \leq t_0$. Then 
\[
\begin{aligned}
c_1 + o(1) = I(u_n) 
&\geq \frac{1}{2} \| u_n \|_2^2 - \int_{\R^N} G(u_n) \, dx
\\
&\geq \frac{1}{4} \| u_n \|_2^2 - \int_{ |u_n| \geq t_0 } G(u_n) \, dx
\\
&\geq \frac{1}{4} \| u_n \|_2^2 - \left( \max_{ |t| \leq \sup_{n} \| u_n \|_\infty } G(t) \right) \left| \left\{ |u_n| \geq t_0 \right\} \right| 
\end{aligned}
\]
and $(u_n)_{n \in \N}$ is bounded in $L^2(\R^N)$. Thus, when $N=1,2$, $(u_n)_{n \in \N}$ is bounded in $X_1$.

Now we use \cref{Prop:decomp} with $a=(N+2)/(2N)$. Then we obtain $K \in \{0,1,2,\dots\} \cup \{\infty\}$, $(y_n^j)_n$, $\wt u_j$ for $0 \leq j < K+1$ such that 
\begin{equation}\label{eq-dec-pos}
	\begin{aligned}
		&u_n(\cdot + y_n^j) \weakto \widetilde{u}_j, 
		\quad u_n(x + y_n^j) \weakto \widetilde{u}_j(x) \quad \text{a.e. $\R^N$}, \\
		&\lim_{n\to\infty} \left( \Psi(u_n) + a \| u_n \|_2^2 \right) \geq \sum_{j=0}^K \left( \Psi(\widetilde{u}_j) + a \| \wt u_j \|_2^2 \right)  , \\
		&\lim_{n \to \infty} \int_{\R^N} G(u_n) \, dx = \sum_{j=0}^K \int_{\R^N} G(\widetilde{u}_j) \, dx, \\
		&\lim_{n \to \infty} \int_{\R^N} g(u_n)u_n \, dx = \sum_{j=0}^K \int_{\R^N} g(\widetilde{u}_j) \widetilde{u}_j \, dx.
	\end{aligned}
\end{equation}
Since $M_1(u_n) = 0$, \eqref{eq-dec-pos} with $a=(N+2)/ (2N)$ yields 
\begin{equation}\label{ineq-M1}
0 = \lim_{n \to \infty} M_1(u_n) 
\geq \sum_{j=0}^K M_1(\wt u_j). 
\end{equation}
When $K=0$ and $\wt u_0 = 0$, from $ \int_{\R^N} G(u_n) \, dx = o(1) = \int_{\R^N} g(u_n) u_n \, dx$ and \eqref{ineq-M1}, it follows that 
$\Psi(u_n) + (N+2)\| u_n \|_2^2/(2N) \to 0$, however, this gives the following contradiction $ 0 < c_1 = \lim_{n \to \infty} I(u_n) = 0$. 
Hence, either $K \geq 1$ or $K=0$ with $\wt u_0 \neq 0$.

As in the proof of \cref{th:1}, set $J := \left\{ j \ : \ \wt u_j \neq 0 \right\}$. 
Then $J \neq \emptyset$ and \eqref{ineq-M1} becomes 
\begin{equation}\label{ineq-M_1J}
0 \geq \sum_{j \in J} M_1 ( \wt u_j ). 
\end{equation}
Therefore, there exists $j_0 \in J$ such that $M_1(\wt u_{j_0}) \leq 0$. By $\wt u_{j_0} \neq 0$, we claim $\theta_{j_0} = \Theta( \Psi( \wt u_{j_0} ) , \wt u_{j_0}  ) \leq 1$. 
Indeed, by recalling $H$ and $\cH$ in the proofs of \cref{Lem-Theta-pos,lem:positive-mass-minimizers-crit}, 
$\wt u_{j_0} \neq 0$ and $M_1(\wt u_{j_0}) \leq 0$, we see that 
\[
0 < \Psi(\wt u_{j_0}) \leq \int_{\R^N} H( \wt u_{j_0} ) \, dx = \cH( \wt u_{j_0} , 1  ).
\]
Therefore, \eqref{eq-diff-cH} leads to 
\[
\frac{d}{d\theta} \int_{\R^N} H(\theta \wt u_{j_0}) \, dx \big|_{\theta = 1} = \frac{\partial}{\partial \theta } \cH( \wt u_{j_0} ,\theta) \big|_{\theta=1} 
> \frac{1}{N} \int_{\R^N} g'(u) u^2 - g(u) u \, dx \geq 0.
\]
Using \eqref{eq-diff-cH} again, we observe that the map $[1,\infty) \ni \theta \mapsto \int_{\R^N} H(\theta \wt u_{j_0}) \, dx$ is strictly increasing. 
Since $\Theta (\Psi (\wt u_{j_0}) , \wt u_{j_0} ) > 0$ is the unique solution to 
\[
\Psi( \wt u_{j_0} ) = \int_{\R^N} H( \theta \wt u_{j_0}) \, dx, 
\]
we get $\theta_{j_0} \leq 1$. Thus by writing $h(s) := g(s)s/2 - G(s) \geq 0$, from 
$(\wt u_{j_0})_{\theta_{j_0}} \in \cM_1$, \eqref{eq:I-on-M1}, \eqref{eq-mono-int}, the weak lower semicontinuity of $\Psi$ and 
Fatou's lemma, it follows that 
\begin{equation*}
	\begin{aligned}
	c_1 \leq I( ( \wt u_{j_0} )_{\theta_{j_0}} ) 
	&= \frac{2 \theta_{j_0}^N }{N+2} 
	\left\{ \Psi(  \wt u_{j_0} ) + \int_{\R^N} h( \theta_{j_0} \wt u_{j_0} ) \, dx \right\}
	\\
	&\leq \frac{2}{N+2} \left\{ \Psi( \wt u_{j_0} ) + \int_{\R^N} h( \wt u_{j_0} ) \, dx \right\}
	\\
	&\leq \frac{2}{N+2} \liminf_{n \to \infty} 
	\left\{ \Psi( u_n (\cdot + y^{j_0}_n) ) + \int_{\R^N} h(u_n (\cdot + y^{j_0}_n)) \, dx \right\}
	\\
	&= \liminf_{n \to \infty} I(u_n) = c_1. 
	\end{aligned}
\end{equation*}
Therefore, $\theta_{j_0} = 1$ and $\wt u_{j_0} \in \cM_1$ is a minimizer. 
Furthermore, by \cref{lem:positive-mass-minimizers-crit}, $\wt u_{j_0}$ is a critical point of $I$.

We next prove $\| u_n( \cdot + y^{j_0}_n) - \wt u_{j_0} \|_{X_1} \to 0$. 
To this end, we first show $\sharp J = 1$. If $\sharp J \geq 2$, then since $M_1(\wt u_{j_0}) = 0$, 
\eqref{ineq-M_1J} becomes 
\[
0 \geq \sum_{ j \in J \setminus \{j_0\}} M_1( \wt u_j ).
\]
Thus, there exists $j_1 \in J \setminus \{j_0\}$ such that $M_1(\wt u_{j_1}) \leq 0$ and we can repeat the above argument to obtain 
\[
\wt u_{j_1} \in \cM_1, \quad c_1 = I ( \wt u_{j_1} ) = \frac{2}{N+2} \left( \Psi( \wt u_{j_1} ) + \int_{\R^N} h( \wt u_{j_1} ) \, dx \right). 
\]
Choose $R_0 \gg1$ so large that 
\[
\frac{c_1}{2} < \frac{2}{N+2} \int_{ B(0,R_0) } 1 -  \sqrt{1 - |\nabla \wt u_{j_k}|^2}  + h( \wt u_{j_k} ) \, dx =: K_{ B(0,R_0)} ( \wt u_{j_k} )  \quad (k=0,1). 
\]
By recalling $| y^{j_0}_n - y^{j_1}_n | \to \infty$ and $h(s) \geq 0$ for each $s \in \R$, 
the weak lower semicontinuity of $\Psi_{B(y,r)}$ and Fatou's lemma give 
\[
\begin{aligned}
	c_1 &< \sum_{k=0}^1 K_{ B( 0 , R_0 ) } \left( \wt u_{j_k}  \right) 
	\\
	&\leq \liminf_{n \to \infty} \sum_{k=0}^1 K_{ B( y^{j_k}_n , R_0 ) } \left( u_n  \right) 
	\\
	&\leq \liminf_{n \to \infty} \frac{2}{N+2} \int_{\R^N} 1 - \sqrt{1 - |\nabla u_n|^2}  + h(u_n) \, dx 
	= \liminf_{n \to \infty} I(u_n) = c_1, 
\end{aligned}
\]
which is a contradiction. Hence, $\sharp J = 1$. 
From \eqref{eq-dec-pos} with $a = (N+2)/(2N)$, \eqref{ineq-M1}, $\sharp J = 1$ and $M_1( \wt u_{j_0} ) = 0$, we infer that 
\[
\Psi(u_n (\cdot + y^{j_0}_n) ) + \frac{N+2}{2N} \| u_n ( \cdot + y^{j_0}_n ) \|_2^2 
\to \Psi( \wt u_{j_0} ) + \frac{N+2}{2N} \| \wt u_{j_0} \|_2^2. 
\]
Combining this with \cref{l:strconv}, we have 
$\| u_{n} (\cdot + y^{j_0}_n) - \wt u_{j_0} \|_{X_1} \to 0$.

We move to the proof that each minimizer has a constant sign and an argument is similar to the above. 
Let $u \in \cM_1$ be any minimizer of $c_1$ and suppose $u_+,u_- \not \equiv 0$. 
Since $0 = M_1(u) = M_1(u^+) + M_1(u^-)$, either $M_1(u_+) \leq 0$ or $M_1(u_-) \leq 0$. 
Suppose $M_1(u_+) \leq 0$. By $u_+ \not \equiv 0$, there exists $\theta_+ \leq 1$ such that $ (u_+)_{\theta_+} \in \cM_1$. 
From \eqref{eq-mono-int}, $h(s) = g(s)s/2 - G(s) \geq 0$, $\Psi(u^-) > 0$ and 
\[
\Psi(u) = \Psi(u_+) + \Psi(u_-), \quad \int_{\R^N} h(u)u \, dx = \int_{\R^N} h(u_+) u_+\,d x + \int_{\R^N} h(u_-)u_- \, dx,
\]
it follows that 
\[
\begin{aligned}
	c_1 \leq I((u_+)_{\theta_+}) 
	&= \frac{2 \theta_+^{N}}{N+2} \left\{ \Psi(u_+) + \int_{\R^N} h(\theta_+u_+) \, dx \right\}
	\\
	&\leq \frac{2 }{N+2} \left\{ \Psi(u_+) + \int_{\R^N} h(u_+) \, dx \right\}
	\\
	&< \frac{2}{N+2} \left\{ \Psi(u) + \int_{\R^N} h(u) \, dx \right\} = I(u) = c_1,
\end{aligned}
\]
which is a contradiction. 
A similar contradiction is obtained if $M_1(u^-) \leq 0$ holds. 
Therefore, either $u_+ \equiv 0$ or $u_- \equiv 0$ holds, 
and the strict sign property follows from the same reasoning as in the zero mass case. 

The existence of a radial minimizer follows as in the proof of \cref{th:1}. 
Indeed, let $u \in \cM_1$ be any minimizer and write $v := u^*$ when $u$ is positive, 
and $v := - u^*$ when $u$ is negative. Then we may verify $M_1(v) \leq 0$, 
by \cref{Lem-Theta-pos}, there exists $\theta \in (0,1]$ such that $v_{\theta} \in \cM_1$. 
Then we may infer from \cref{lem:nehari-type-inequality} and $I(w) = \int_{\R^N} f(w)w \, dx / N$ for any $w \in \cM_1$ that 
\[
0<c_1 \leq I( v_{\theta} ) = 
\frac{\theta^N}{N} \int_{\R^N} f(\theta v) \theta v \, dx 
\leq \frac{1}{N} \int_{\R^N} f( v) v \, dx 
= \frac{1}{N} \int_{\R^N} f(u) u \, dx = I(u) = c_1.
\]
Thus, $\theta = 1$ and $v \in \cM_1$ is a radial minimizer. 
\end{proof}

\section{Nonradial solutions}
\label{sec:nonradial}

In this section, \cref{th:4} is proved. Throughout this section, we suppose $N \geq 4$ and 
fix any $k_1,k_2 \in \N$ so that $k_1,k_2 \geq 2$ and $N - k_1 - k_2 \geq 0$ and consider 
\[
\cO = \cO(k_1) \times \cO(k_2) \times \{\mathrm{id}\} \subset \cO(N). 
\]
By writing $x=(x_1,x_2,x_3) \in \R^{k_1} \times \R^{k_2} \times \R^{N-k_1-k_2}$ (when $N=k_1 +k_2$, we write $x=(x_1,x_2)$), 
recall $X_\cO$, $\tau$ and $X_\tau$: 
\[
X_\cO := \left\{ u \in X \ : \ u(\rho x) = u (x) \quad \text{for all $x \in \R^N$ and $\rho \in \cO$} \right\}
\]
and 
\[
(\tau u) (|x_1|,|x_2|,x_3) := - u( |x_2|, |x_1|, x_3 ), \quad 
X_\tau := \left\{ u \in X_\cO \ : \  \tau u = u \right\}, 
\]
where $X = X_0$ in the zero mass case and $X=X_1$ in the positive mass case. 
Since $f$ (resp. $g$) is odd, $\Psi(-u) = \Psi(u)$ and $\Phi(-u) = \Phi(u)$. 
According to the principle of symmetric criticality (see \cite{KoOt04,BIMM}), 
any critical point of $I$ in $X_\tau$ is also a critical point in $X_\cO$, 
and by applying the principle again, it is a critical point in $X$. 
Thus, it suffices to find a critical point of $I|_{X_\tau}$.

We first prove that for the profile decomposition of $(u_n)_{n \in \N}$ in $X_\cO$, 
we may assume $y_n^j \in \{0\} \times \{0\} \times \R^{N-k_1-k_2}$:

\begin{Lem}\label{lem:lionsNR}
	Let $(u_n)_{n \in \N} \subset X_\cO$ be bounded and $K \in \{0\} \cup \N \cup \{\infty\}$, $(y_{n}^j)_{n \in \N} \subset \R^N$, $\wt u_j \in X$ for $ 0 \leq j < K+1$ 
	be as in \cref{Prop:decomp}. Then for each $j$, $\wt u_j \in X_\cO$ and 
	\[
	\limsup_{n \to \infty} \left\{ | y_{n,1}^j | + |y_{n,2}^j |  \right\} < \infty.
	\]
	In particular, we may suppose $y_{n,1}^j = 0$ and $y_{n,2}^j = 0$ for any $1 \leq j < K +1$ and $n$. 
%
\end{Lem}

\begin{proof}
The assertion $\wt u_j \in X_\cO$ follows from the closedness of $X_{\cO}$ in $X$. 
The latter assertion may be proved as in \cite{Wi96}. 
Let $j \in \N$. 
By $u_n( \cdot + y_n^j ) \rightharpoonup \wt u_j \not \equiv 0$ weakly in $X$ and $u_n (\cdot + y_n^j) \to \wt u_j$ in $L^{p_0}_{\rm loc} (\R^N)$, 
there exist $r_j>0$ and $N_j \in \N$ such that for all $n \geq N_j$, 
$\| u_n \|_{L^{p_0} (B_{r_j}(y_n^j)) } \geq \| \wt u_j \|_{L^{p_0} (B_{r_j}(0))}/ 2 > 0$. 
Since $k_1,k_2 \geq 2$, one may find $m_1 = m_1( k_1, |y_{n,1}^j| )$ (resp. $m_2 = m_2(k_2,|y_{n,2}^j|)$) disjoint balls $B_1,\dots, B_{m_1}$ (resp. $B_1,\dots, B_{m_2}$) 
with radius $r_j$ and centered on $\partial B_{|y_{n,1}^j|} (0) \cap \R^{k_1} \times \{0\} \times \{0\}$ 
(resp. $\partial B_{|y_{n,1}^j|} (0) \cap \{0\} \times \R^{k_2} \times \{0\}$), 
where $m(k_i,r) \to \infty$ as $r \to \infty$ for $i=1,2$. 
Since $(u_n)_{n \in \N}$ is bounded in $X$ and $u_n \in X_{\cO}$, for each $n \geq N_j$ and $i=1,2$,  
\[
\infty > \sup_{k \in \N} \| u_k \|_{p_0}^{p_0} \geq \sum_{\ell=1}^{m_i} \| u_n \|_{L^{p_0} (B_\ell) }^{p_0} 
= m_i \| u_n \|_{L^{p_0} (B_{r_j} (0)) } \geq \frac{m_i}{2} \| \wt u_j \|_{L^{p_0} (B_{r_j}(0)) } > 0.
\]
Since $m_i = m_i(k_i,r) \to \infty$ as $r \to \infty$, $(y^j_{n,i})_{n \in \N}$ $(i=1,2)$ are bounded (the bound depends on $ \| \wt u_j \|_{L^{p_0} (B_{r_j} (0))} $ ) 
and we complete the proof. 
\end{proof}

\begin{proof}[Proof of \cref{th:4}]
We first prove the existence of minimizer corresponding to $\inf_{\cM_\tau} I$, 
where $\cM_\tau := \cM \cap X_\tau$.  
Let $(u_n)_{n \in \N} \subset \cM_{\tau}$ satisfy $I(u_n) \to \inf_{\cM_{\tau}} I \in [ \inf_{\cM} I , \infty )$. 
As in the proofs of \cref{th:1,th:3}, $(u_n)_{n \in \N}$ is bounded in $X$. 
By \cref{Prop:decomp} and \cref{lem:lionsNR}, 
there exist $K \in \{0\} \cup \N \cup \{\infty\}$, $(y^j_n)_{n \in \N} \subset \{0\} \times \{0\} \times \R^{N-k_1-k_2}$ and $\wt u_j$ such that
\begin{equation*}
	\begin{aligned}
		&u_n(\cdot + y_n^j) \rightharpoonup \wt u_j, \quad u_n(x+y_n^j) \to \wt u_j(x) \quad \text{a.e. $\R^N$},
		\\
		& \lim_{n \to \infty} \Psi(u_n) \geq \sum_{j=0}^K \Psi_0(\wt u_j) \quad 
		\left( \text{resp.} \ \lim_{n \to \infty} \left\{ \Psi(u_n) + a \left\| u_n \right\|^2_2 \right\} \geq \sum_{j=0}^K \left\{ \Psi( \wt u_j) + a \| \wt u_j \|_2^2 \right\} \right) 
		\\
		& \lim_{n \to \infty} \int_{\R^N} F(u_n) \, dx = \sum_{j=0}^K \int_{\R^N} F(\wt u_j) \, dx 
		\quad \left( \text{resp.} \ \lim_{n \to \infty} \int_{\R^N} G(u_n) \, dx = \sum_{j=0}^K \int_{\R^N} G(\wt u_j) \, dx \right),
		\\
		& \lim_{n \to \infty} \int_{\R^N} f(u_n)u_n \, dx = \sum_{j=0}^K \int_{\R^N} f(\wt u_j) \wt u_j \, dx 
		\quad \left( \text{resp.} \ \lim_{n \to \infty} \int_{\R^N} g(u_n) u_n \, dx = \sum_{j=0}^K \int_{\R^N} g(\wt u_j) \wt u_j \, dx \right).
	\end{aligned}
\end{equation*}
Then as in the proof of \cref{th:1,th:3}, we may prove that $\inf_{\cM_{\tau}} I $ is attained.

Let $u \in \cM_{\tau}$ satisfy $I(u) = \inf_{\cM_{\tau}} I$. 
Then we may use the argument in the proof of \cref{lem:minimizers-are-critical,lem:positive-mass-minimizers-crit} 
to verify that $u$ is a critical point of $I$ on $X_{\tau}$. 
As pointed in the beginning of this section, $u$ is a critical point of $I$ in $X$, and hence 
\cref{l:RegCritP} implies that $u$ is a classical solution of \eqref{eq:BI}.

When $k_1 = k_2$, we next prove $I(u) \geq 2\inf_{\cM} I$ where $u \in \cM_\tau$ is a minimizer corresponding to $\inf_{\cM_{\tau}} I$. 
Set 
\[
\begin{aligned}
	\Omega_1&:=\{x = (x_1,x_2,x_3) \in\R^N \ : \  |x_1|>|x_2|\},\\
	\Omega_2&:=\{x=(x_1,x_2,x_3) \in\R^N \ : \  |x_1|<|x_2|\}.
\end{aligned}
\]
In this case, $\tau$ is identified with $\tau(x_1,x_2,x_3) = (x_2,x_1,x_3)$ 
and notice that $\tau \Omega_1 = \Omega_2$ and $| \R^N \setminus (\Omega_1 \cup \Omega_2) | = 0$. 
Since $u\in X_\tau \setminus \{0\}$ and $u = 0$ on $\partial \Omega_i$ $(i=1,2)$, 
we get $\bm{1}_{\Omega_i} u\in X$ ($i=1,2$) and $0 = M(u) = 2M( \bm{1}_{\Omega_i} u )$. 
Thus, $\bm{1}_{\Omega_i} u \in \cM$ and 
\[
I(u)=I(\bm{1}_{\Omega_1} u)+I(\bm{1}_{\Omega_2} u)=2I(\bm{1}_{\Omega_1} u)\geq 2 \inf_{\cM} I.
\]

Finally, we show $I(u) > 2 \inf_\cM I$. From $\bm{1}_{\Omega_1} u \in \cM$, 
if $I(\bm{1}_{\Omega_1} u) = \inf_{\cM} I$, then by \cref{th:1,th:3}, either $\bm{1}_{\Omega_1} u > 0$ or else $\bm{1}_{\Omega_1} u < 0$. 
However, this is a contradiction. Thus, $I(\bm{1}_{\Omega_1} u) > \inf_{\cM} I$ and $I(u) = 2 I(\bm{1}_{\Omega_1} u) > 2 \inf_{\cM} I$ holds. 
\end{proof}

\section*{Acknowledgements}

B. Bieganowski was partly supported by the National Science Centre, Poland (grant no. 2023/51/B/ST1/00968). J. Mederski was partly supported by the National Science Centre, Poland (grant no. 2025/57/B/ST1/04804). 
This work was supported by JSPS KAKENHI Grant Number JP24K06802. 
This work was initiated and partially done when N. Ikoma visited Warsaw and he would like express his gratitude 
to University of Warsaw and Institute of Mathematics of the Polish Academy of Sciences. 

%

\end{document}